\documentclass[11pt,letterpaper]{amsart}
\usepackage{graphicx}				
\usepackage{amssymb}

\usepackage{pst-node}
\usepackage{tikz-cd} 
\usepackage{mathtools}

\usepackage{mathrsfs}

\usepackage{amsthm}

\usepackage{amsfonts}
\usepackage{yhmath}
\usepackage{cleveref}
\usepackage{caption}
\usepackage{float}
\theoremstyle{definition}
\newtheorem{definition}{Definition}[section]

\theoremstyle{remark}
\newtheorem{remark}[definition]{Remark}

\theoremstyle{theorem}
\newtheorem{theorem}[definition]{Theorem}

\theoremstyle{corollary}
\newtheorem{corollary}[definition]{Corollary}

\theoremstyle{lemma}
\newtheorem{lemma}[definition]{Lemma}

\theoremstyle{example}
\newtheorem{example}[definition]{Example}

\theoremstyle{remarks}
\newtheorem{remarks}[definition]{Remarks}

\theoremstyle{prop}
\newtheorem{prop}[definition]{Proposition}
\author{}
\makeatletter
\@namedef{subjclassname@2020}{%
	\textup{2020} Mathematics Subject Classification}
\makeatother
\begin{document}
\title{Stability conditions on cyclic categories I: Basic definitions and examples}

\author{Yucheng Liu}

\address{Beijing International Center for Mathematical Research, Peking University, No.5 Yi-
heyuan Road Haidian District, Beijing, 100871, P.R.China}
\email{liuyucheng@bicmr.pku.edu.cn}



\maketitle 
\begin{abstract}
A triangulated category $\mathcal{C}$ with a canonical Bott's isomorphism $[2]\xrightarrow{\sim}id$ is called a cyclic category in this paper. We give a new notion of stability conditions on a $k$-linear Krull-Schmidt cyclic category.  Given such a stability condition $\sigma$, we can assign a Maslov index to each basic loop in such a  category. If all Maslov indices vanish, we get $\widetilde{\mathcal{C}},\widetilde{\sigma}$ as the $\mathbb{Z}$-lifts of $\mathcal{C},\sigma$ respectively, such that $\widetilde{\mathcal{C}}$ is a $\mathbb{Z}$-graded triangulated category and $\widetilde{\sigma}$ is a Bridgeland stability condition on $\widetilde{\mathcal{C}}$. Moreover, we showed that there is an isomorphism $$Stab_{cyc}^{0,e}(\mathcal{C})\xrightarrow{\simeq} Stab(\widetilde{\mathcal{C}})/2\mathbb{Z},$$ where $Stab_{cyc}^{0,e}(\mathcal{C})$ denotes the equivalence classes of stability conditions which are deformation equivalent to $\sigma$, and $Stab(\widetilde{\mathcal{C}})$ denotes the space of Bridgeland stability conditions on $\widetilde{\mathcal{C}}$.

We provide some examples of stability conditions on a cyclic category. We also discuss some interesting phenomena in these examples, such as the chirality symmetry breaking phenomenon and nontrivial monodromy. The chirality symmetry breaking phenomenon involves  stability conditions which can not be lifted to Bridgeland stability conditions.
	
\end{abstract}

\section{Introduction}

The phenomenon $[2]=[0]$ has a relatively long history in mathematics, and can be found in many different branches of  mathematics and physics. For examples: the celebrated Tate cohomology of cyclic groups and Bott's periodicity; in early 1980s, Eisenbud discovered that every finitely generated $S$ module admits a free resolution which will eventually become 2-periodic, where $S$ is the algebra of functions of a hypersurface with an isolated singularity (see \cite{Homologicalalgebraoveracompleteintersection}); at almost the same time, people were developing the theory of cyclic homology (e.g. \cite{Noncommutativedifferentialgeometry},\cite{Cyclichomologyderivationsandthefreeloopspace},\cite{Excisioninbivariantperiodiccycliccohomolgy},\cite{Ontopologicalcyclichomology}), the periodic cyclic homology played a role in such a theory; Kapustin and Li used the category of matrix factorizations to describe the D-branes in Landau-Ginzburg models by following a proposal of Kontsevich (see e.g.  \cite{DbranesinLGmodelsandalgebraicgeometry},\cite{TopologicalcorrelatorsinLGmodels}), which was further studied by mathematicians (e.g. \cite{Compactgeneratorsformatrixfactorizations}, \cite{KLformularevisit}, \cite{CherncharactersandHRRformatrixfactorizations}, \cite{Residueanddualityforsingularitycategories}); we could also find the phenomenon of [2]=[0] in Treumann's paper \cite{SmiththeoryandgeometricHeckealgebras}.

On the other hand, the story of stability conditions on triangulated categories is relatively new.  Motivated by Douglas's work on D-branes and $\Pi$ stability (see e.g. \cite{douglas2002dirichlet}), Bridgeland introduced a general theory of stability conditions on triangulated categories in \cite{bridgeland2007stability}; the theory was further studied by Kontsevich and Soibelman in \cite{kontsevich2008stability}. 

Let $\mathcal{C}$ be a $k$-linear  triangulated category, if there exists a canonical isomorphism $\beta:[2]\simeq id$ between two functors, it is easy to see that there is no $t$-structures on $\mathcal{C}$. Hence no Bridgeland stability conditions exists on $\mathcal{C}$ either. However, as Bridgeland stability condition can be viewed as an $\mathbb{R}$-grading refinement of a $\mathbb{Z}$-grading ($t$-structure) on a triangulated category, we expect that there exists a notion of $S^1$-grading which refines a $\mathbb{Z}/2\mathbb{Z}$-grading of a cyclic category. 

Unlike on the real line, there are many homotopy non-equivalent paths to connects two points on $S^1$, so we need to introduce a new data to distinguish these paths. This new data is the degree function of a real decomposition on $\mathcal{C}$ (see Definition \ref{real decomposition} for the precise definition). 

Roughly speaking, given a real decomposition on $\mathcal{C}$ and any two indecomposable objects $E,F\in\mathcal{C}$, we have following decomposition  $$Hom_{\mathcal{C}}(E,F)=\bigoplus_{a\in\mathbb{R}}Hom_a(E,F),$$ which satisfy some natural conditions. A morphism $f\in Hom_a(E,F)$ is called a homogeneous morphism of degree $a$. We get a degree function $$q:\{Nontrivial \ homogeneous \ morphims \}\rightarrow \mathbb{R}.$$ This enables us to define the notion of connecting path (see Definition \ref{connecting path}) and liftable commutative diagram (see Definition \ref{liftable and locally liftable diagrams}), which are the basic notion in our definitions and results. We use these notion to define what is a stability condition on a cyclic category in Section 3.

In Section 4, we define what is a basic loop and its Maslov index. This notion of Maslov index plays a similar role as its namesake in Fukaya categories. Indeed, we have the following lifting theorem.

\begin{theorem}\label{lifting theorem}
	Given a stability condition $\sigma=(\mathcal{Q},Z,\phi, q)$ on a $k$-linear Krull-Schmidt cyclic category $\mathcal{C}$, if  the Maslov indices of  all basic loops are zero. Then there are  $\mathbb{Z}$-lifts of $\sigma$ and $\mathcal{C}$, which we denote by $\widetilde{\sigma}$ and $\widetilde{\mathcal{C}}$ respectively, such that $\widetilde{\sigma}$ is a Bridgeland stability condition on $\widetilde{\mathcal{C}}$.
\end{theorem}

Here the  data $(Z,\phi, q)$ in a stability condition $\sigma$ is called a charge triple, and in fact, the  $\mathbb{Z}$-lift $\widetilde{\mathcal{C}}$ only depends on the charge triple. We say that $\mathcal{C}$ is liftable with respect to a charge triple $R\coloneqq (Z,\phi, q)$ if the condition in Theorem \ref{lifting theorem} is satisfied. Furthermore, one can define when two charge triples are deformation equivalent (see Definition \ref{Definition of liftable and deformation equivalence}).  And we proved the following result in Section 4. 

\begin{prop}\label{main version of connection}
	Let $\mathcal{C}$ be a $k$-linear Krull-Schmidt connective cyclic category,  $R_1=(Z_1,\phi_1,q_1)$ and $R_2=(Z_2,\phi_2, q_2)$ be two charge triples on $\mathcal{C}$. Suppose that $R_1,R_2$ are deformation equivalent and $\mathcal{C}$ is  liftable with respect to $R_1$. Then $\mathcal{C}$ is also liftable with respect to $R_2$.  
	
	Moreover,  if we denote the $\mathbb{Z}$-lifts of $\mathcal{C}$ with respect to $R_1, R_2$ by $\widetilde{\mathcal{C}}_1$ and $\widetilde{\mathcal{C}}_2$ respectively, there exists an  equivalence (not canonical)  $H:\widetilde{\mathcal{C}}_1\simeq \widetilde{\mathcal{C}}_2$ such that the following diagram is commutative. $$\begin{tikzcd}
		\widetilde{\mathcal{C}}_1 \arrow{rd}[swap]{\pi_1} \arrow{rr}{H}& & \widetilde{\mathcal{C}}_2 \arrow {ld}{\pi_2} \\ & \mathcal{C} & 
	\end{tikzcd}$$ 
\end{prop}

Given a charge triple $R$ such that $\mathcal{C}$ is liftable with respect to it, a consistent choice of these equivalences in Proposition \ref{main version of connection} forms a connection of $\mathbb{Z}$-lifts of $\mathcal{C}$ fibered over the set of charge triples which are deformation equivalent to $R$. Hence Theorem \ref{lifting theorem} and Proposition \ref{main version of connection} provide us a map $$Stab_{cyc}^0(\mathcal{C})\rightarrow Stab(\widetilde{\mathcal{C}}),$$ where $Stab_{cyc}^0(\mathcal{C})$ consists of stability conditions whose charge triples are deformation equivalent to $R$, and $\widetilde{\mathcal{C}}$ is the associated $\mathbb{Z}$-lift of $\mathcal{C}$ with respect to $R$. We use $Stab(\widetilde{\mathcal{C}})$ to denote the space of Bridgeland stability conditions on $\widetilde{\mathcal{C}}$.

And a close look at this map provides us the following comparison theorem, which is proved in Section 5.

\begin{theorem}\label{Main theorem in introduction}
There is an isomorphism $$Stab_{cyc}^0(\mathcal{C})/\simeq \xrightarrow{} Stab(\widetilde{\mathcal{C}})/2\mathbb{Z}, $$ where  the equivalence relation is defined in Definition \ref{equivalent relations} and the $2\mathbb{Z}$ action is given by $k\mapsto [k]$ for any $k\in 2\mathbb{Z}$. 
\end{theorem}

 However, there are many stability conditions which can not be lifted to Bridgeland stability conditions.  In Section 6, we provide some examples of stability conditions on the category of $\mathbb{Z}/3\mathbb{Z}$-equivariant matrix factorizations of $w=x^{3}$. In this section, we also discuss some interesting phenomena in these examples, for instance, the chirality symmetry breaking phenomenon and the non-trivial monodromy of the map from strong stability conditions to their central charges. The chirality symmetry breaking phenomenon involves  stability conditions which can not be lifted to Bridgeland stability conditions.

\subsection{Outline of this paper} In Section 2, we briefly review some basic definitions and results of matrix factorizations. In Section 3, we firstly give some preliminary definitions, such as real decompositions, connecting paths, liftable commutative diagrams and so on. Then we give our definition of stability conditions on cyclic categories in the end of this section. In Section 4, we introduce our notion of  basic loop and its Maslov index, and prove the lifting theorem in this section. In Section 5, we prove the uniqueness of Harder-Narasimhan filtration under the assumption that all Maslov indices are non-negative. Furthermore, we prove the comparison theorem. In Section 6, we give the examples of stability conditions on the category of $\mathbb{Z}/3\mathbb{Z}$-equivariant matrix factorizations of $w=x^{3}$. We also discuss the chirality symmetry breaking phenomenon and nontrivial monodromy in these examples.

\subsection{Acknowledgement} I would like to thank Emanuel Scheidegger for teaching me the background in physics, many helpful discussions and his guidance to the literature. I  thank Alex Martsinkovsky and Amnon Neeman for answering my questions via email correspondences. I also thank Tom Bridgeland and Yu Qiu for helpful comments. I thank Qingyuan Bai, Hanfei Guo, Gang Han, Chunyi Li, Zhiyu Liu, Yongbin Ruan, Zhiyu Tian, Shizhuo Zhang and Xiaolei Zhao for  helpful discussions. Theorem \ref{Main theorem} is motivated by Zhiyu Tian's question. The final write-up of this paper was done while the author was visiting Zhejiang University, whose hospitality is gratefully acknowledged. 

\subsection{Notation and convention} A cyclic category, usually denoted by $\mathcal{C}$ in this paper, is always assumed to be $k$-linear, Krull-Schmidt and essentially small.

\section{Matrix factorizations}

\subsection{Matrix Factorizations}

Matrix factorizations was firstly introduced by Eisenbud in \cite{Homologicalalgebraoveracompleteintersection}. It got the attention from physicists because of a proposal of Kontsevich, which suggests using them to describe the D-branes on Landau-Ginzburg model (see e.g.  \cite{DbranesinLGmodelsandalgebraicgeometry} and \cite{TopologicalcorrelatorsinLGmodels}).  Mathematicians further studied them (see e.g. \cite{Compactgeneratorsformatrixfactorizations}, \cite{CherncharactersandHRRformatrixfactorizations}, \cite{Residueanddualityforsingularitycategories}, \cite{Orlovtriangulatedcatofsingularities}, and \cite{GradedMF}). All the materials in this section are  directly taken from these references.

Let us briefly recall the story of matrix factorizations. Suppose that  $(R,m)$ is a regular local ring and $M$ is a finitely generated $R$-module. The famous Auslander-Buchsbaum formula is $$pd(M)=dim(R)-depth(M).$$ In particular, if the depth of $M$ is equal to the Krull dimension of $R$, then $M$ is free.

However, if we consider a ring $S=R/w$ of hypersurface of singularity, where $w$ is singular at $m$. The Auslander-Buchsbaum formula fails, and the condition $$depth(M)=dim(S)$$ no longer implies that $M$ is free. A module satisfying such condition is called a maximal Cohen-Macaulay module.

Given a maximal Cohen-Macaulay $S$-module $M$, we can consider $M$ as an $R$-module. Then by Auslander-Buchsbaum formula we know that there is a length 1 $R$-free resolution of $M$. Suppose the following short exact sequence $$0\rightarrow F^1\xrightarrow{f} F^0\rightarrow M\rightarrow 0$$ is the $R$-free resolution of $M$. Since multiplication by $w$ on $M$ is trivial, there exists a homotopy $g$  such that the diagram $$\begin{tikzcd}
	F^1\arrow{r}{f}\arrow{d}{w}  & F^0\arrow{dl}[swap]{g}\arrow{d}{w} \\ F^1 \arrow{r}{f}  & F^0
\end{tikzcd}$$ commutes. We use the pair $(f,g)$ to define a matrix factorization of $w$. Note that the original maximal Cohen-Macaulay $M$ is isomorphic to $coker(f)$.

Therefore, it is natural to bring up the following definition, where we assume that $R$ is a commutative ring over a field $k$ and $w\in R$ is an element.
\begin{definition}\label{Definition of MF}
	A matrix factorization of the potential $w$ over $R$  is a pair 
	
	$$(E,\delta_E)=(E^0\xrightarrow{\delta_0}E^1\xrightarrow{\delta_1} E_0),$$
	
\begin{itemize}
	\item $E=E^0\oplus E^1$ is a $\mathbb{Z}/2$-graded finitely generated projective $R$-module, and 
	\item $\delta_E=\begin{pmatrix}
	0 & \delta_1 \\ \delta_0 & 0
	\end{pmatrix}\in End_R^1(E)$ is an odd (i.e. of degree $1\in\mathbb{Z}/2$) endomorphism of $E$, such that $\delta_E^2=w\cdot id _E$.
	
\end{itemize}
\end{definition}
\begin{remark}
	The map  $\delta_E$ is usually called a twisted differential in the literature. We will adopt this terminology in this paper.  Note that $\delta_E^2=w\cdot id _E$ is also equivalent to $\delta_0\delta_1=\delta_1\delta_0=w\cdot I$. 
	
	In the case when $R=k[[x_1,x_2,\cdots x_n]]$ is the ring of formal power series in $n$ variables over a field $k$, $E_0$ and $E_1$ are free $R$-modules. Moreover,  $E_0, \ E_1$ are of the same rank over $R$ by the requirement $\delta_0\delta_1=\delta_1\delta_0=w\cdot I$, hence all the maps $\delta_0$, $\delta_1$ and $\delta_E$ can be represented as matrices of elements in $R$. For all explicit examples in this paper, we will take $R$ to be a  ring of formal power series.
\end{remark}

To a potential $w\in R$ we can associated a $\mathbb{Z}/2$ dg-category $MF(w)\coloneqq MF(R,w)$ whose objects are matrix factorizations of $w$ over $R$. The morphisms from $\bar{E}=(E,\delta_E)$ to $\bar{F}=(F,\delta_F)$ are elements of the $\mathbb{Z}/2$-graded module of $R$-linear homomorphisms $$\mathcal{H}om_w(\bar{E},\bar{F})\coloneqq Hom_{Mod_R}(E,F)=Hom_{\mathbb{Z}/2-Mod_R}(E,F)\oplus Hom_{\mathbb{Z}/2-Mod_R}(E,F[1]).$$

The $\mathbb{Z}/2$-graded dg-structure is given by the following differential on $f\in \mathcal{H}om_w(\bar{E},\bar{F})$ $$df=\delta_F\circ f-(-1)^{|f|}f\circ \delta_E.$$ 

We use $$HMF(R,w)=H^0MF(R,w)$$ to denote the associated homotopy category of $MF(R,w)$, i.e, the space of morphisms in this category are chain maps up to homotopy. The homotopy category $HMF(R,w)$ is naturally triangulated (see e.g. \cite{Orlovtriangulatedcatofsingularities}) with the shift functor $$T: (E^0\xrightarrow{\delta_0}E^1\xrightarrow{\delta_1} E_0)\mapsto(E^1\xrightarrow{-\delta_1}E^0\xrightarrow{-\delta_0} E_1).$$

\begin{remark}\label{equivalence of three categories}
	Eisenbud proved that (see \cite{Homologicalalgebraoveracompleteintersection}) in the case when $(R,m)$ is a regular local ring and $S=R/w$, where $w$ is singular at the closed point $m$, the functor $coker$ induces an equivalence $$coker: HMF(R,w)\xrightarrow{\sim} \underline{MCM}(S).$$ Here $\underline{MCM}(S)$ is the stable category of maximal Cohen-Macaulay $S$-modules. The objects are maximal Cohen-Macaulay modules, the morphisms are defined by $$\underline{Hom}_S(M,M')=Hom_S(M,M')/P,$$ where $P$ denotes the set of $S$-linear homomorphisms factoring through some free $S$-module.

	In such a case, Buchweitz proved that there is another equivalence $$\underline{MCM}(S)\rightarrow D_{sing}^b(S),$$ where $D_{sing}^b(S)$ is the Verdier quotient $$D_{sing}^b(S)\coloneqq D^b(S)/D^b_{perf}(S).$$
	
	Here $D^b(S)$ is the derived category of all complexes of $S$-modules with finitely generated total cohomology. Such a complex is called perfect if it is isomorphic in $D^b(S)$ to a bounded complex of free $S$-modules. The full triangulated subcategory formed by the perfect complexes is denoted by $D^b_{perf}(S)$. See also the paper \cite{Orlovtriangulatedcatofsingularities} for the proofs of such equivalences.
	
	Moreover, in such a case, there is a non-degenerate pairing on the morphism spaces in $HMF(R,w)$. This beautiful formula was firstly introduced by Kapustin and Li using the path integral method when $k=\mathbb{C}$ (see  e.g. \cite{TopologicalcorrelatorsinLGmodels}). This formula was mathematically proved in \cite{Residueanddualityforsingularitycategories} and \cite{KLformularevisit}.
\end{remark}
	
	 For simplicity, we can assume $R=k[[x_1,\cdots, x_n]]$. For the explicit meaning of the notation in the formula, please consult \cite[Sections 1-3]{KLformularevisit}.

	\begin{theorem}\cite[Theorem 3.4]{KLformularevisit}
		The pairing $$Hom(X, Y)\otimes_R Hom(Y,X[n])\rightarrow k,$$ $$(F,G)\mapsto (-1)^{\binom{n+1}{2}}\frac{1}{n!}Res[\substack{tr(FG(dQ)^n) \\ \partial_1w, \partial_2w,\cdots, \partial_nw}]=(-1)^{\binom{n+1}{2}}\frac{1}{n!}Res[\substack{tr(GF(dP)^n) \\ \partial_1w, \partial_2w,\cdots, \partial_nw}]$$provides a homologically non-degenerate pairing on the morphism complexes of the category $MF(R,w)$ associated to the germ of an isolated hypersurface singularity. Here $P, Q$ are the twisted differentials of $X$ and $Y$ respectively.
	\end{theorem}
	
	\begin{proof}
		The formula involving the twisted differential $Q$ of $Y$ is proved in \cite{KLformularevisit}, and the formula involving  the twisted differential $P$ of $X$ can be proved by the same method.
	\end{proof}
	
We conclude this section by briefly recalling the $G$-equivariant version of Matrix factorizations (see e.g. \cite[Section 2.1]{CherncharactersandHRRformatrixfactorizations}). Indeed, if $G$ is a finite group of automorphisms of $R$ which fixes  the potential $w$, one defines the $G$-equivariant $\mathbb{Z}/2$-graded dg-category of matrix factorizations $MF_G(w)$ (and the corresponding homotopy category $HMF_G(w)$), by requiring that all modules and morphisms should be $G$-equivariant. In other words, the module $E$ in Definition \ref{Definition of MF} should be a $\mathbb{Z}/2$-graded finitely generated projective $R$-module equipped with a compatible $G$-action, and $\delta_E$ has to be $G$-equivariant. Morphisms between $G$-equivariant factorizations $\bar{E}$ and $\bar{F}$ should also be compatible with the action of $G$, i.e. $$Hom_{MF_G(w)}(\bar{E},\bar{F})=Hom_{MF(w)}(\bar{E},\bar{F})^G.$$

\section{Stability conditions on cyclic categories}

\subsection{Bridgeland stability conditions}

	The theory of Bridgeland stability conditions was introduced by Bridgeland in \cite{bridgeland2007stability}, motivated by Douglas's work on D-branes and $\Pi$-stability \cite{douglas2002dirichlet}. This theory was further studied by Kontsevich and Soibelman in \cite{kontsevich2008stability}.
	In this section, we will review some basic notions in the theory of stability conditions (see \cite{beilinson1982faisceaux}, \cite{bridgeland2007stability}, \cite{kontsevich2008stability} and \cite{bayer2017stability}).
	
	The first notion is $t$-structures on triangulated categories, which was firstly introduced in \cite{beilinson1982faisceaux}. 
	\begin{definition}
		Let $\mathcal{D}$ be an triangulated category. A $t$-structure on $\mathcal{D}$ is a pair of full subcategories $(\mathcal{D}^{\leq 0},\mathcal{D}^{\geq 0})$ satisfying the condition (i), (ii) and (iii) below. We denote $\mathcal{D}^{\leq n}=\mathcal{D}^{\leq 0}[-n]$, $\mathcal{D}^{\geq n}=\mathcal{D}^{\geq 0}[-n]$ for every $n\in\mathbb{Z}$. Then the conditions are:
		
		(i) $Hom(E,F)=0$ for every $E\in\mathcal{D}^{\leq 0}$ and $F\in\mathcal{D}^{\geq 1}$;
		
		(ii) $\mathcal{D}^{\leq -1}\subset \mathcal{D}^{\leq 0}$ and  $\mathcal{D}^{\geq 1}\subset \mathcal{D}^{\geq 0}$.
		
		(iii) every object $E\in\mathcal{D}$ fits into an exact triangle 
		
		$$\tau^{\leq 0}E\rightarrow E\rightarrow \tau^{\geq 1}E\rightarrow \cdots$$ with $\tau^{\leq 0}E\in\mathcal{D}^{\leq 0}$, $\tau^{\geq 1}E\in\mathcal{D}^{\geq 1}$.
		
		The heart of the $t$-structure is $\mathcal{A}=\mathcal{D}^{\leq 0}\cap\mathcal{D}^{\geq 0}$. It is an abelian category (see \cite[Theorem 8.1.9]{hotta2007d}). The associated cohomology functors are defined by $H^0(E)=\tau^{\leq 0}\tau^{\geq 0}E$, $H^i(E)=H^0(E[i])$. 
	\end{definition}
	
	Combining this definition with Harder-Narasimhan filtrations, Bridgeland defined the notion of stability conditions on a triangulated category in \cite{bridgeland2007stability}.
	
	\begin{definition}\label{slicing}
		A Bridgeland stability condition $(\mathcal{P},Z)$  on a triangulated category $\mathcal{D}$ consists of a group homomorphism $Z:K_0(\mathcal{D})\rightarrow \mathbb{C}$, which factors through  a fixed group homomorphism $K_0(\mathcal{D})\rightarrow \Lambda$, where $\Lambda$ a lattice of finite rank. This group homomorphism is called a central charge. And full subcategories $\mathcal{P}(\phi)\in\mathcal{D}$ for each $\phi\in\mathbb{R}$, satisfying the following axioms:
		
		\par
		
		(a) if $E\in\mathcal{P}(\phi)$ is a nonzero object, then $Z(E)=m(E)exp(i\pi\phi)$ for some $m(E)\in\mathbb{R}_{>0}$,
		\par
		(b) for all $\phi \in \mathbb{R}$, $\mathcal{P}(\phi+1)=\mathcal{P}(\phi)[1]$,
		
		\par
		(c) if $\phi_1>\phi_2$ and $A_j\in\mathcal{P}(\phi_j)$ then $Hom_{\mathcal{D}}(A_1,A_2)=0$,
		
		\par

		(d) for every $0\neq E\in\mathcal{D}$ there exist a finite sequence of real numbers
		
		$$\phi_1>\phi_2>\cdots>\phi_m$$and a sequence of morphisms 
		
		$$0=E_0\xrightarrow{f_1}E_1\xrightarrow{f_2} \cdots \xrightarrow{f_m}E_m=E $$such that the cone of $f_j$ is in $\mathcal{P}(\phi_j)$ for all $j$.
	\end{definition}

	\begin{remark}

		(i) If we allow $m(E)$ to be $0$ for $\phi\in\mathbb{Z}$ in (a), then the pair $(\mathcal{P},Z)$ is called a weak stability condition. In \cite{kontsevich2008stability}, the authors require the pair $(\mathcal{P},Z)$ to satisfy one extra condition (support property) to be a stability condition. We do not include this condition because it is not needed in this paper. 
		
		(ii) This notion of Bridgeland stability conditions is categorical. Indeed, suppose that  $H:\mathcal{D}_1\xrightarrow{\sim}\mathcal{D}_2$ is an exact equivalence between two triangulated categories, and $\sigma=(\mathcal{P},Z)$ is a Bridgeland stability condition on $\mathcal{D}_2$. Then we have a Bridgeland stability condition $H^*\sigma=(H^*\mathcal{P},H^*Z)$ on $\mathcal{D}_1$, which is defined in the following way: \begin{itemize}
			\item $H^*\mathcal{P}(\phi)=\{E| H(E)\in\mathcal{P}(\phi)\}$, for any $\phi\in\mathbb{R}$.
			\item $H^*Z(E)=Z(H(E))$ for any $E$ in $\mathcal{D}_1$.
		\end{itemize} It is easy to check this is a Bridgeland stability condition on $\mathcal{D}_1$.
		
	\end{remark}
	The data $\mathcal{P}$ of full subcategories $\mathcal{P}(\phi)$ is called a slicing on $\mathcal{D}$, a slicing can be viewed as a refinement of a  $t$-structure on a triangulated category. Indeed, one can easily check that a slicing on $\mathcal{D}$ gives us a lot of $t$-structures on $\mathcal{D}$: for any $\phi\in\mathbb{R}$, we have a $t$-structure $(\mathcal{P}(>\phi-1),\mathcal{P}(\leq \phi))$ on $\mathcal{D}$. 
	
	In particular, we get a heart $\mathcal{P}(0,1]=\mathcal{P}(>0)\cap \mathcal{P}(\leq 1)$. Hence, a  stability condition $(\mathcal{P},Z)$ gives us a pair $(\mathcal{A},Z)$, where $\mathcal{A}$ is an abelian category. This construction results in an equivalent definition of stability conditions (see \cite[Proposition 5.3]{bridgeland2007stability}).
	
	In this paper, we are interested in the triangulated categories with the special property $[2]\simeq [0]$. Hence we have the following definition.
 \begin{definition}
 	A triangulated category $\mathcal{C}$ is called a cyclic category if there is canonical isomorphism between two functors $$\beta: [2]\simeq id.$$ We usually call $\beta$ a Bott's isomorphism.
 \end{definition}

It is easy to see that there exist no $t$-structures on any nontrivial cyclic categories. Indeed, assume that there is a $t$-structure on $\mathcal{C}$. Then $[1]\simeq [-1]$ implies that $id_{A[1]}\in Hom(A[1], A[1])\simeq Hom(A[1], A[-1])=0$ for any object $A$ in the heart, which implies that the heart is trivial.  As a slicing is a refinement of a $t$-structure, any nontrivial cyclic categories do not admit any Bridgeland stability conditions.

However, if one intuitively think a slicing as a helix parametrization of the structure  of a triangulated category, one would expect to have the notion of $S^1$-gradings on cyclic categories, which is a circle parametrization of the structure  of a cyclic category (see Figure 1 for this intuition).

\begin{figure}[ht]\label{Figure 1}
	\flushleft
	\includegraphics[scale=1.0]{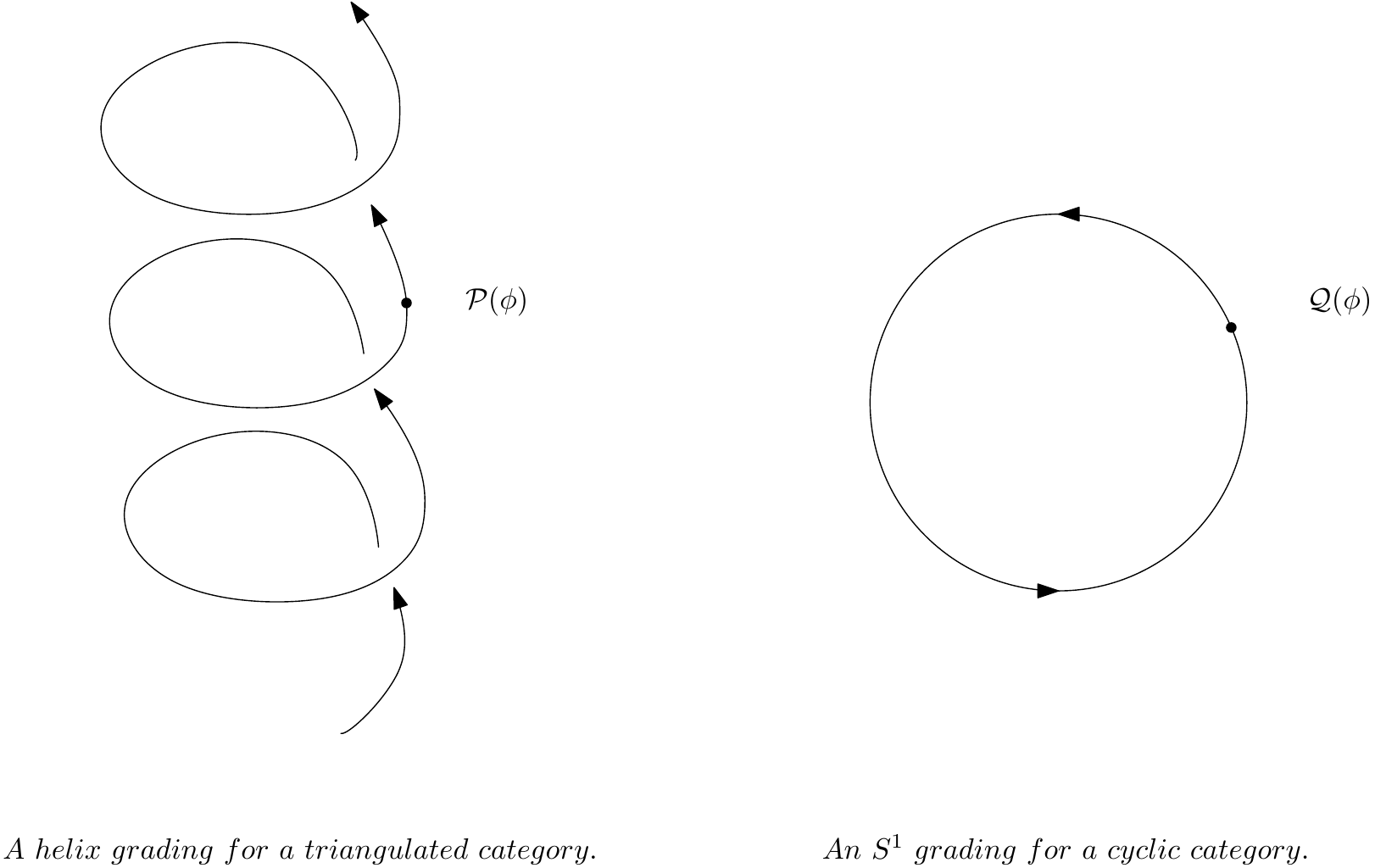}
	\caption{}
\end{figure}

\subsection{Real decompositions and liftable commutative diagrams} To make this intuition precise, we need to do some preliminary work.

From here on, we will assume that the $k$-linear cyclic category $\mathcal{C}$ is Krull-Schmidt,  i.e.  every object decomposes into direct sum of indecomposable objects in a unique way up to isomorphism, and the endomorphism ring of an indecomposable object is local. 
\begin{definition}\label{real decomposition}
	Let $\mathcal{C}$ be a $k$-linear Krull-Schmidt cyclic category. We say that $\mathcal{C}$ admits a real decomposition if every space $Hom_{\mathcal{C}}(E,F)$ admits a grading by $\mathbb{R}$, i.e.  $$Hom_{\mathcal{C}}(E,F)=\bigoplus_{a\in\mathbb{R}}Hom_a(E,F)$$ for any two indecomposable objects $E,F\in\mathcal{C}$. And these gradings should  satisfy the following conditions. \begin{enumerate} 
		\item For any morphism  $f\in Hom_a(E,F)$, we also have $ f[1]\in Hom_a(E[1],F[1])$. 
		
		\item For any two morphisms $f,g$ with $$f\in Hom_a(E,F), \ g\in Hom_b(F,G),$$  we have  $$g\circ f\in Hom_{(a+b)}(E,G).$$

		\item For any indecomposable object $E\in\mathcal{C}$, the morphisms $id_E\in Hom_0(E,E)$ and $\beta_E\in Hom_0(E[2],E)$.
		
		\item  A morphism $f\in Hom_{a}(E,F)$ is called a homogeneous morphism of degree $a$.  For any isomorphism $f$, if $f$ is homogeneous, its degree is $0$.
	\end{enumerate}

This decomposition also defines a function $$q:\{Nontrivial \ homogeneous \ morphims \}\rightarrow \mathbb{R},$$ which sends a nontrivial homogeneous morphism to its degree. The function $q$ is called the degree function associated with the real decomposition of $\mathcal{C}$.

\end{definition}

\begin{remarks} \label{Remarks after real decomposition}
	 (i) The degree function $q$ is also called $R$-charge of homogeneous morphism in physics literature (see e.g. \cite{StabilityofLGbranes}).  The real decomposition endows a graded local algebra structure on $Hom_{\mathcal{C}}(E,E)$ for any indecomposable object $E\in\mathcal{C}$, and a graded $Hom_{\mathcal{C}}(E,E)-Hom_{\mathcal{C}}(F,F)$ bi-module structure on $Hom_{\mathcal{C}}(E,F)$ for any two indecomposable objects $E,F$.
	 
	 	
	 
	
	(ii) If $\mathcal{C}$ is just a $k$-linear Krull-Schmidt category, we can define a similar notion of $G$-decomposition, where $G$ is an arbitrary  group instead of $\mathbb{R}$.  In this case, we only need the decomposition to satisfy conditions (2) and (4) in Definition \ref{real decomposition}. And all the results in this subsection holds in this generality.
 \end{remarks}
We have the following lemma from the definition of real decomposition.
\begin{lemma}
If $\mathcal{C}$ admits a real decomposition and $E,F$ are two isomorphic indecomposable objects, then there exists an isomorphism $f\in Hom_0(E,F)$. And $f$ induces an isomorphism between $\mathbb{R}$-graded vector spaces $$Hom_{\mathcal{C}}(F,G)\xrightarrow{\circ f}Hom_{\mathcal{C}}(E,G).$$
\end{lemma}

\begin{proof}
	Since $\mathcal{C}$ is Krull-Schmidt. The sum of two non-invertible morphisms in $Hom(E,F)$ is still not invertible. Hence any isomorphism must have a homogeneous summand $f$ which is also an isomorphism. By condition (4) in Definition \ref{real decomposition}, $f$ is of degree $0$. Condition (2) in Definition \ref{real decomposition} implies the lemma.
\end{proof}

\begin{remark}
	Take $G$ to be $E$ in the lemma, we get that $f^{-1}$ is also a homogeneous morphism of degree $0$.
\end{remark}

Given a real decomposition of $\mathcal{C}$, we can define quasi-homogeneous morphisms and liftable quasi-homogeneous morphisms.

\begin{definition}
	Suppose a morphism $f:A\rightarrow B$ can be written in the following way  $$f:\bigoplus_{i=1}^l A_i\rightarrow \bigoplus_{j=1}^m B_j,$$ where $A_i, B_j$ are nontrivial indecomposable objects for all $i,j$. Then $f$ is called quasi-homogeneous if all its direct summand $f_{ji}:A_i\rightarrow B_j$ are homogeneous morphisms for arbitrary $i,j$.
	
\end{definition}

We also need the following terminology to define liftable quasi-homogeneous morphisms.
\begin{definition}\label{connecting path}
	
	  (1) Given a real decomposition on $\mathcal{C}$, we call the following diagram $$E_0\xleftrightarrow{f_0}E_1\xleftrightarrow{f_1}\cdots\xleftrightarrow{f_{n-1}}E_n $$ a connecting path from $E_0$ to $E_n$ if  $E_i$ is an indecomposable object for any  $0\leq i\leq n $ and $f_j$ is a nontrivial homogeneous morphism in either $Hom_{\mathcal{C}}(E_j,E_{j+1})$ or $Hom_{\mathcal{C}}(E_{j+1},E_{j})$ for any $0\leq j\leq n-1$.  We usually use $l$ to denote a connecting path, and let $$q(l)\coloneqq \Sigma_{i=0}^{n-1}Sign(f_i)q(f_i),$$ where $$Sign(f_i)=\begin{cases} 1, & \text{if} f_i\in Hom_{\mathcal{C}}(E_i,E_{i+1});\\ -1 & \text{if}
f_i\in Hom_{\mathcal{C}}(E_{i+1},E_{i}).
	\end{cases}$$
	In the special case when $E_0=E_n$, we say that $l$ is a connecting loop. A connecting path $l$ is called simple if it contains no connecting loops.
	
	(2) The cyclic category $\mathcal{C}$ is called connective if for any two indecomposable objects $E,F$, there exists a connecting path between them.
	
	(3) Let $\textbf{D}$ be a commutative digram in $\mathcal{C}$, if all the morphisms in $\textbf{D}$ are quasi-homogeneous. One could say that a connecting path $$l:E_0\xleftrightarrow{f_0}E_1\xleftrightarrow{f_1}\cdots\xleftrightarrow{f_{n-1}}E_n $$ is in $\textbf{D}$ if any homogeneous morphism $f_i$ is a homogeneous direct summand of a quasi-homogeneous morphism in that commutative diagram.

\end{definition}

\begin{remarks}
	(1) The degree of a connecting path depends on its direction. Indeed, a connecting path $l$ from $E$ to $F$ can also be viewed as a connecting path from $F$ to $E$, which we denote it by $\bar{l}$.  We have $$q(l)=-q(\bar{l}).$$
	
	(2) The property of being connective is independent of the real decomposition, it is a property of the category $\mathcal{C}$ itself.
	
(3) For	any commutative digram $\textbf{D}$ in $\mathcal{C}$, there is an associated directed graph $G(\textbf{D})$ whose vertices are objects in $\textbf{D}$, and edges are the morphisms in $\textbf{D}$. We are only interested in commutative diagrams whose associated graph is simple, i.e. with neither self loops nor multiple edges. So every commutative diagram in this paper is assumed to be of such kind.
\end{remarks}

All diagrams below are assumed to be in a $k$-linear Krull-Schmidt cyclic category $\mathcal{C}$ which admits a real decomposition. We introduce two special kinds of such commutative diagrams .

\begin{definition}\label{liftable and locally liftable diagrams}
	Let  $\textbf{D}$ be a commutative diagram  in $\mathcal{C}$. We say that $\textbf{D}$ is liftable if the following conditions are satisfied: \begin{enumerate}
		\item All the morphisms in $\textbf{D}$ are quasi-homogeneous.
		
		\item For any connecting loop $l$ in $\textbf{D}$, we have $q(l)=0$.
	\end{enumerate}

	A commutative diagram $\textbf{D}$ is called locally liftable, if for every simply connected sub-diagram $\textbf{D}'\subset \textbf{D}$, we have that $\textbf{D}'$ is liftable.

	A liftable commutative diagram $\textbf{D}$ is called connective, if for any two nontrivial indecomposable summands  $A, B$ of the objects in $\textbf{D}$, there is a connecting path $l:A\dashrightarrow B$ in $\textbf{D}$.
\end{definition}
\begin{remarks}\label{Lifting quasi-homogeneous morphism and degree matrix}
(1) A sub-diagram in $\textbf{D}$ is a commutative diagram $\textbf{D}'$ whose associated graph $G(\textbf{D}')$ is a sub-graph of $G(\textbf{D})$, and $\textbf{D}', \textbf{D}$ share the same object and morphism over the same vertex and edge respectively. The sub-diagram is called simply connected if the associated graph $G(\textbf{D}')$ is simply connected.

(2) Let $\textbf{D}$ be a liftabe commutative diagram and $f:\bigoplus_{i=1}^l A_i\rightarrow \bigoplus_{j=1}^m B_j$ be a quasi-homogeneous morphism in $\textbf{D}$. We can define a degree matrix $q(f)$ of $f$ in $\textbf{D}$. Indeed, if there exists a  connecting path $l$ from $A_i$ to $B_j$ in $\textbf{D}$, we define the $ji-th$ entry of $q(f)$ to be $q(l)$. Otherwise,  this entry is not defined. 
	
	Note that the degree matrix $q(f)$ depends on the liftable commutative diagram $\textbf{D}$. We suppress this dependence in the notation.
	
\end{remarks}

\begin{example}

Let $$f=\begin{pmatrix}
	f_{11} & \cdots & f_{1l} \\ \cdots & \cdots  &  \cdots \\ f_{m1} & \cdots & f_{ml}
\end{pmatrix}   $$ be a liftable homogeneous morphism. If  moreover $f$ is connective, i.e. for arbitrary $1\leq i\leq l, 1\leq j\leq m$, there exist a connecting path from $A_i$ to $B_j$. Every entry in the degree matrix $q(f)$ is well-defined, and  if we denote $$R(f)\coloneqq \begin{pmatrix} exp(2\pi i\cdot q(f)_{11}) & \cdots & exp(2\pi i\cdot q(f)_{1l}) \\ \cdots & \cdots & \cdots \\ exp(2\pi i\cdot q(f)_{m1})) &\cdots & exp(2\pi i \cdot q(f)_{ml})
	
	\end{pmatrix}.$$
	It is easy to show that the matrix $R(f)$ is of rank 1. We usually call $R(f)$ the $R$-matrix of $f$ when $f$ is a connective liftable quasi-homogeneous morphism. 
	
\end{example}




We have the following simple lemma.
\begin{lemma}\label{Composition lemma}
	For any chain $\textbf{D}$ of quasi-homogeneous morphisms $$A_0\xrightarrow{f_0} A_1\xrightarrow{f_1}\cdots \xrightarrow{ f_{n-1}} A_n.$$ If the chain $\textbf{D}$ is liftable, the following  commutative digram $$\begin{tikzcd}
		& A_1\arrow{r}{f_1} & A_2 \arrow{r}{f_2} &\cdots \arrow{r}{f_{n-2}} & A_{n-1} \arrow{rd}{f_{n-1}}\\ A_0 \arrow{ru}{f_0}  \arrow{rru}[swap]{f_1\circ  f_0}  \arrow{rrrru}[swap]{f_{n-2}\circ\cdots\circ f_0}\arrow{rrrrr}[swap]{f_{n-1}\circ\cdots\circ  f_0}& & & & & A_n
	\end{tikzcd}$$ is also liftable.
\end{lemma}
\begin{proof}
For any $1\leq i\leq n$, we can write $$A_0=\bigoplus_{j=1}^m A_{0,j},\ \  A_i=\bigoplus_{k=1}^l A_{i,k}.$$ Let $f_{kj}: A_{0,j}\rightarrow A_{i,k}$ to denote the corresponding summand in $f_{i-1}\circ\cdots\circ f_0$. The morphism $f_{kj}$ can be written as the sum of homogeneous morphisms from $A_{0,j}$ to $A_{i,k}$ in $\textbf{D}$.  As $\textbf{D}$ is liftable, these homogeneous morphisms are of the same degree. Hence any nontrivial morphism $f_{kj}$ is a homogeneous morphism of degree equal to $q(l)$, where $l:A_{0,j}\dashrightarrow A_{i,k}$ is a connecting path in $\textbf{D}$.

One can easily show that the commutative diagram is liftable  by this observation. 
\end{proof}

\begin{remark}
	In fact, the same argument can show that the associated $(n+1)$-complete commutative digram is also liftable.
\end{remark}

\begin{lemma}\label{Factorization lemma}
		Consider the following  diagram $$\begin{tikzcd}
		\bigoplus_{i=1}^lA_i \arrow{r}{f}  \arrow{d}{g}&\bigoplus_{j=1}^m B_j \\ \bigoplus_{k=1}^n C_k
	\end{tikzcd}$$ where $$f=\begin{pmatrix}
		f_{11} & \cdots & f_{1l} \\ \cdots & \cdots  &  \cdots \\ f_{m1} & \cdots & f_{ml}
	\end{pmatrix}  \ \ \ \ \ \   g=\begin{pmatrix}
		g_{11} & \cdots & g_{1l} \\ \cdots & \cdots  &  \cdots \\ g_{n1} & \cdots & g_{nl}
	\end{pmatrix}$$ are quasi-homogeneous morphisms. Suppose that this diagram is a sub-digram of a liftable commutative diagram $\textbf{D}$, and can be completed into a commutative diagram by a morphism $h$. $$\begin{tikzcd}
	\bigoplus_{i=1}^lA_i \arrow{r}{f}  \arrow{d}{g}&\bigoplus_{j=1}^m B_j \arrow[ld,dotted, "h"]\\ \bigoplus_{k=1}^n C_k &
\end{tikzcd}$$ Then there is a quasi-homogeneous morphism $\bar{h}$ such that if we add $\bar{h}$ to the digram $\textbf{D}$, we get a liftable digram $\bar{\textbf{D}}$ and the following diagram $$\begin{tikzcd}
\bigoplus_{i=1}^lA_i \arrow{r}{f}  \arrow{d}{g}&\bigoplus_{j=1}^m B_j \arrow[ld,dotted, "\bar{h}"]\\ \bigoplus_{k=1}^n C_k &
\end{tikzcd}$$ is  a liftable commutative diagram. 
\end{lemma}

\begin{proof}
	If we write $$h=\begin{pmatrix}
		h_{11} &\cdots & h_{1m} \\ \cdots & \cdots &\cdots \\ h_{n1} &\cdots & h_{nm}
	\end{pmatrix}$$ we claim that if we take $$\bar{h}=\begin{pmatrix}
	\bar{h}_{11} &\cdots & \bar{h}_{1m} \\ \cdots & \cdots &\cdots \\ \bar{h}_{n1} &\cdots & \bar{h}_{nm}
\end{pmatrix}$$ where $$ \bar{h}_{kj}=\begin{cases} \text{Summand of}\ h_{kj}\  \text{in degree} \ q(l), & \text{if there is a connecting path}\  l:B_j\dashrightarrow  C_k\in\textbf{D};\\ 0, & \text{otherwise}. \end{cases} $$ Then $\bar{h}$ satisfy the statement in the lemma.  

First of all, the new diagram  is  lifatable by construction. Indeed, for any connecting loop $l$ in $\bar{\textbf{D}}$, we can replace the segments $\bar{h}_{kj}$ in $l$ by a connecting path $l_{kj}:B_j\dashrightarrow  C_k\in\textbf{D}$. This does not change the degree of the loop by the definition of $\bar{h}_{kj}$. Hence we get a new loop $l'\in\textbf{D}$ with $q(l)=q(l')=0$.

The next thing is to check that $g=\bar{h}\circ f$. It suffices to prove that \begin{equation*}
	\begin{pmatrix}
	\bar{h}_{k1} & \cdots &\bar{h}_{km}
\end{pmatrix}\begin{pmatrix}
	f_{11} & \cdots & f_{1l} \\ \cdots & \cdots  &  \cdots \\ f_{m1} & \cdots & f_{ml}
\end{pmatrix}=\begin{pmatrix}
g_{k1} & \cdots & g_{kl}
\end{pmatrix} \end{equation*} for any $1\leq k\leq n$.  

For any given $1\leq i_1\leq l$, we let $$J\coloneqq \{j_1,j_2,\cdots ,j_s\}=\{1\leq j\leq m| f_{ji_1}\neq 0\}.$$ The case $J=\emptyset$ is trivial. We assume that $J\neq \emptyset$. If there is no connecting paths from $B_{j_r}$ to $C_k$ in $\textbf{D}$ for any $1\leq r\leq s$, the morphism $g_{ki_1}$ must be $0$. Hence $$\Sigma_{j=1}^m\bar{h}_{kj}f_{ji_1}=0=g_{ki_1}.$$

On the other hand, if there is an integer $1\leq r_1\leq s$ such that there is a connecting path  from $B_{j_{r_1}}$ to $C_k$ in $\textbf{D}$. Then there is a connecting path $$B_{j_r}\xleftarrow{f_{j_ri_1}} A_{i_1}\xrightarrow{f_{j_{r_1}i_1}} B_{j_{r_1}}\dashrightarrow C_k$$ for any $1\leq r\leq s$. Hence we have that $\bar{h}_{kj_r}f_{j_ri_1}$ are homogeneous morphisms of same degree for any $1\leq r\leq s$, and this degree equals to $q(g_{ki_1})$ if $g_{ki_1}\neq 0$. This implies that $\Sigma_{j=1}^m\bar{h}_{kj}f_{ji_1}=g_{ki_1}$ if $g_{ki_1}\neq 0$. If $g_{ki_1}=0$, we know that $\Sigma_{j=1}^m\bar{h}_{kj}f_{ji_1}$ is a homogeneous summand of $\Sigma_{j=1}^m h_{kj}f_{ji_1}=g_{ki_1}=0$, which is also $0$. This completes the proof.

\end{proof}

\begin{remark}
 One thing worth noticing is that, as we can see in the proof,  the morphism $\bar{h}$ may depend on the ambient commutative liftable diagram $\textbf{D}$. 
\end{remark}

\begin{prop}\label{Factorization proposition}
	Suppose the following diagram $$\begin{tikzcd}
		& A_1\arrow{r}{f_1} &\cdots \arrow{r}{f_{i-1}} & A_i & A_{i+1} \arrow{r}{f_{i+1}} & \cdots\arrow{r}{f_{n-2}} & A_{n-1}\arrow{rd}{f_{n-1}} & \\A_0\arrow{ru}{f_0} \arrow{rrrrrrr}{g}&&&&&&& A_n
	\end{tikzcd}$$ is a sub-diagram of a liftable commutative diagram $\textbf{D}$, and it can be completed into a commutative diagram by a morphism $h$. $$\begin{tikzcd}
	& A_1\arrow{r}{f_1} &\cdots \arrow{r}{f_{i-1}} & A_i \arrow[r,dotted,"h"]& A_{i+1} \arrow{r}{f_{i+1}} & \cdots\arrow{r}{f_{n-2}} & A_{n-1}\arrow{rd}{f_{n-1}} & \\A_0\arrow{ru}{f_0} \arrow{rrrrrrr}{g}&&&&&&& A_n
\end{tikzcd}$$

Then there is a quasi-homogeneous morphism $\bar{h}$ such that the following diagram $$\begin{tikzcd}
		& A_1\arrow{r}{f_1} &\cdots \arrow{r}{f_{i-1}} & A_i \arrow[r,dotted,"\bar{h}"]& A_{i+1} \arrow{r}{f_{i+1}} & \cdots\arrow{r}{f_{n-2}} & A_{n-1}\arrow{rd}{f_{n-1}} & \\A_0\arrow{ru}{f_0} \arrow{rrrrrrr}{g}&&&&&&& A_n
\end{tikzcd}$$ is a liftable commutative diagram.
\end{prop}

\begin{proof}
	The proof is similar to the proof of Lemma \ref{Factorization lemma}, we sketch the proof for readers' convenience.
	
	We construct $\bar{h}$ and $\overline{f_{n-1}\circ\cdots\circ f_{i+1}\circ h}$ as in the proof of Lemma \ref{Factorization lemma}. By Lemma \ref{Composition lemma} and Lemma \ref{Factorization lemma}, we know that the following diagram $$\begin{tikzcd}
		& A_1\arrow{r}{f_1} &\cdots \arrow{r}{f_{i-1}} & A_i \arrow{rrrrd}[swap]{\overline{f_{n-1}\circ\cdots\circ f_{i+1}\circ h}}& A_{i+1} \arrow{r}{f_{i+1}} & \cdots\arrow{r}{f_{n-2}} & A_{n-1}\arrow{rd}{f_{n-1}} & \\A_0\arrow{ru}{f_0} \arrow{rrrrrrr}{g}&&&&&&& A_n
	\end{tikzcd}$$ is liftable and commutative. It suffices to prove that $\overline{f_{n-1}\circ\cdots\circ f_{i+1}\circ h}=f_{n-1}\circ\cdots\circ f_{i+1}\circ\bar{h}$. We let $f\coloneqq f_{n-1}\circ\cdots\circ f_{i+1}$.

As in the proof of Lemma \ref{Factorization lemma}, we can assume that $A_{i},A_n$ are indecomposable objects. Hence we can write $$A_{i+1}=\bigoplus_{j=1}^m D_{j},\ \ h=\begin{pmatrix}
	h_1\\ \cdots\\ h_m
\end{pmatrix}, \ \  f=\begin{pmatrix}
f_1 & \cdots, f_m
\end{pmatrix}.$$ Let $$J\coloneqq \{j_1,j_2,\cdots ,j_s\}=\{1\leq j\leq m| f_{j}\neq 0\}.$$ This implies that there is connecting path form $D_{j_r}$ to $A_n$ in $\textbf{D}$ for any $1\leq r\leq s$. Thus the argument in proof of Lemma \ref{Factorization lemma} shows that $\Sigma_{j=1}^m f_j\bar{h_j}$ is the homogeneous summand of $\Sigma_{j=1}^m f_jh_j$ in the right degree. Hence $f\circ \bar{h}=\overline{f\circ h}$.
\end{proof}
\begin{corollary}\label{Fill-in corollary}
		Let $\textbf{D}$ be any liftable commutative diagram, suppose that we can add a morphism $h$ into $\textbf{D}$ such that the new diagram $\textbf{D}'$ is still commutative. Then there exists a quasi-homogeneous morphism $\bar{h}$ such that, if we add $\bar{h}$ into $\textbf{D}$ in the same position, the new diagram $\bar{\textbf{D}}$ is a lifatable commutative diagram.
\end{corollary}
\begin{proof}
We can construct $\bar{h}$ as in the proof of Lemma \ref{Factorization lemma}, the new diagram is obviously liftable. And the commutativity follows from Lemma \ref{Composition lemma} and Proposition \ref{Factorization proposition}.
\end{proof}

Sometimes, we need to glue two liftable commutaitve diagrams along a common sub-diagram. In general the glued diagram will not be liftable. However, we have following easy lemma.

\begin{lemma}
	Let $\textbf{D}_1$ and $\textbf{D}_2$ be two liftable diagrams, $\textbf{D}_3$ be a common sub-diagram of $\textbf{D}_1$ and $\textbf{D}_2$, and $\textbf{D}$ be the glued diagram of  $\textbf{D}_1$ and $\textbf{D}_2$ along $\textbf{D}_3$. If $\textbf{D}_3$ is connective, then $\textbf{D}$ is liftable.
\end{lemma}
\begin{proof}
	We can focus on the loops which are not in $\textbf{D}_1$ or $\textbf{D}_2$. Any loop of such kind $l:E\dashrightarrow E$ in $\textbf{D}$ can be written as $l_1\circ l_2\circ \cdots \circ l_n$ where $l_{i}:E_i\dashrightarrow E_{i+1}$ is a connecting path in $\textbf{D}_1$ or $\textbf{D}_2$ depending on the parity of $i$, and the indecomposable objects $E_i$ are in the diagram $\textbf{D}_3$ with $E_0=E_{n+1}=E$.
	
	As $\textbf{D}_3$ is connective, we have connecting paths $l_i':E_i\dashrightarrow E_{i+1}$ in $\textbf{D}_3$. Therefore, one can get $$q(l)=\Sigma_{i=1}^n q(l_i)=\Sigma_{i=1}^n q(l_i')=q(l_1'\circ l_2'\circ\cdots\circ l_n')=0.$$
	The second equation follows from the assumption that $\textbf{D}_1$ and $\textbf{D}_2$ are liftable. Hence, the lemma is proved.
\end{proof}

\subsection{Definition of stability conditions}
The following definition is about  the compatibility between the triangulated structure of $\mathcal{C}$ and the real decomposition on $\mathcal{C}$. 

\begin{definition}\label{preliftability}
	Let $\mathcal{C}$ be a $k$-linear Krull-Schmidt cyclic category which admits a real decomposition. We say that $\mathcal{C}$ is pre-liftable with respect to the real decomposition if the following conditions hold: \begin{enumerate}
		\item for any liftable quasi-homogeneous morphism $f:A\rightarrow B$, we can complete it to a distinguished triangle $$A\xrightarrow{f}  B\xrightarrow{g} C\xrightarrow{h} A[1]$$ such that the following diagram $$ \cdots \xrightarrow{g[-1]} A[-1]\xrightarrow{h[-1]} A\xrightarrow{f} B\xrightarrow{g}C\xrightarrow{h}A[1]\xrightarrow{f[1]} B[1]\xrightarrow{g[1]} \cdots $$
		is  liftable,
		\item if we have the following liftable commutative diagram $$\begin{tikzcd}
			X\arrow{d}{u} \arrow{r}{id} & X\arrow{d}{v\circ u} & & \\ Y\arrow{r}{v} \arrow{d}{j} &Z\arrow{r}{l}\arrow{d}{m} & X'\arrow{r}{i} & Y[1] \\ Z'\arrow{d}{k} & Y' \arrow{d}{n}& &  \\ X[1]\arrow{r}{id} & X[1] &&
		\end{tikzcd}$$where all row and  columns and are distinguished triangle, then it can be completed into the following liftable commutative diagram $$\begin{tikzcd}
		X\arrow{d}{u} \arrow{r}{id} & X\arrow{d}{v\circ u} & & \\ Y\arrow{r}{v} \arrow{d}{j} &Z\arrow{r}{l}\arrow{d}{m} & X'\arrow{r}{i} \arrow{d}{id}& Y[1]\arrow{d}{j[1]} \\ Z'\arrow{r}{f}\arrow{d}{k} & Y'\arrow{r}{g} \arrow{d}{n}& X' \arrow{r}{h}& Z[1] \\ X[1]\arrow{r}{id} & X[1] &&
	\end{tikzcd}$$ where all the columns and rows are distinguished triangles.
	\end{enumerate}
	
\end{definition}
Now, we are ready to define stability conditions on a $k$-linear Krull-Schmidt cyclic category.
\begin{definition}\label{stability condition on cyclic categories}
	A stability condition on a $k$-linear Krull-Schmidt cyclic category $\mathcal{C}$ consists of four parts $(\mathcal{Q}, Z, \phi, q)$, where $\mathcal{Q}$ is a circle slicing, i.e. full subcategories $\mathcal{Q}(\psi)$ for any $\psi\in (0,2]$, a central charge $Z:K_0(\mathcal{C})\xrightarrow{v}\Lambda\rightarrow \mathbb{C}$,   a map $\phi$ sending every indecomposable object $E$ to its phase $\phi(E)\in(0,2]$, and $q$ a degree function of a real decomposition of $\mathcal{C}$. Here $\mathcal{C}$ is pre-liftable with respect to the real decomposition and these data  satisfy the following compatibility conditions. 
	
	\begin{enumerate}
		\item We have that $\phi(E[1])\equiv \phi(E)+1 (mod\ 2\mathbb{Z})$.
		
		\item For any indecomposable object $E$,  the central charge can be written as $$Z(E)=m(E)e^{i\pi\phi(E)},$$ where $m(E)\in\mathbb{R}_{\geq 0}$ and $\phi(E)\in(0,2] $. 
		
		\item  For any  homogeneous morphism $f:E_1\rightarrow E_2$ between two indecomposable objects,  we have $$q(f)\equiv \phi(E_2)-\phi(E_1) (mod\ 2\mathbb{Z}).$$
		
		\item $\mathcal{Q}(\psi)[1]=\mathcal{Q}(\psi')$, where $\psi'\equiv \psi+1\ (mod \ 2\mathbb{Z})$ and $\psi, \psi'\in(0,2]$.
		
		\item For any object $E\in\mathcal{Q}(\phi)$,  we  have $m(E)>0$ and $\phi(E)=\phi$.
		\item For any nontrivial homogeneous morphism $f:E_1\rightarrow E_2$,  if  $E_k\in\mathcal{Q}(\phi_k)$ for $k=1,2$, we have $q(f)\geq 0$.
		
		\item For any indecomposable object $E\in \mathcal{C}$,  we have the following filtration:
		
		$$\begin{tikzcd}[column sep=large]
			& E_0 \arrow{d}{f_1}&E_1\arrow{d}{f_2} & \cdots & E_{n-2} \arrow{d}{f_{n-1}} &  E_{n-1}\arrow{d}{f_n}\\ 
				0=E_0  \arrow{ru}{id} & E_1\arrow{d}{p_1} \arrow{ru}{id} & E_2\arrow{d}{p_2}\arrow{ru}{id} &\cdots  \arrow{ru}{id}& E_{n-1}\arrow{ru}{id} \arrow{d}{p_{n-1}} &E_n=E \arrow{d}{p_n}\\ & Q_1 & Q_2& \cdots & Q_{n-1} & Q_n
				\end{tikzcd} $$
	\\\\ 
	such that it satisfies the following conditions:\begin{itemize}
		
		\item the whole diagram  is  connective and liftable,

		\item for any $1\leq i\leq n$, we have $Q_i\in\mathcal{Q}(\phi_i)$ being nontrivial semi-stable objects,
		
			\item the diagram can be completed into the following liftable diagram, 
		$$\begin{tikzcd}[column sep=large]
			& \cdots \arrow{d}{t_1[-1]} & \cdots \arrow{d}{t_2[-1]} & \cdots & \cdots \arrow{d}{t_{n-1}[-1]} & \cdots \arrow{d}{t_n[-1]}\\
			& E_0 \arrow{d}{f_1}&E_1\arrow{d}{f_2} & \cdots & E_{n-2} \arrow{d}{f_{n-1}} &  E_{n-1}\arrow{d}{f_n}\\ 
			0=E_0  \arrow{ru}{id} & E_1\arrow{d}{p_1} \arrow{ru}{id} & E_2\arrow{d}{p_2}\arrow{ru}{id} &\cdots  \arrow{ru}{id}& E_{n-1}\arrow{ru}{id} \arrow{d}{p_{n-1}} &E_n \arrow{d}{p_n}\\ & Q_1\arrow{d}{t_1} & Q_2\arrow{d}{t_2}& \cdots & Q_{n-1}\arrow{d}{t_{n-1}} & Q_n\arrow{d}{t_n}\\ & E_0[1] \arrow{d}{f_1[1]} &E_1[1] \arrow{d}{f_2[1]} &\cdots & E_{n-2}[1]\arrow{d}{f_{n-1}[1]} & E_{n-1}[1]\arrow{d}{f_n[1]} \\ & E_1[1]\arrow{d}{p_1[1]}\arrow{ru}{id} & E_2[1] \arrow{ru}{id} \arrow{d}{p_2[1]}& \cdots \arrow{ru}{id} & E_{n-1}[1]\arrow{ru}{id} \arrow{d}{p_{n-1}[1]}& E_n[1]\arrow{d}{p_n[1]}\\ & \cdots &  \cdots & \cdots & \cdots & \cdots
		\end{tikzcd} $$ where the sequences $$E_i\xrightarrow{f_{i+1}} E_{i+1}\xrightarrow{p_{i+1}}Q_{i+1}\xrightarrow{t_{i+1}} E_{i}[1]$$ are distinguished triangles for all $0\leq i\leq n-1$. We call this diagram the Harder-Narasimhan diagram of $E$,
		
		\item for any $2\leq i\leq n$, there exists a real number $c_i<0$ such that for any two indecomposable summands $Q_{i-1,1}$ and $Q_{i,1}$ of $Q_{i-1}$ and $Q_{i}$ respectively, there exists a simple connecting path from  $Q_{i-1,1}$ to $Q_{i, 1}$ in the diagram, and we have
		$$q(l)=c_i<0$$ any such simple connecting path $l$.

	\end{itemize}
	\end{enumerate}


\end{definition}

	\begin{remarks}\label{Remarks under the definition of stability conditions}
	(i) Note that in condition (2), instead of simply requiring that $m(E)>0$ for all indecomposable objects, we define $\phi(E)$ even when $Z(E)=0$, this is because we want condition (3) to hold even when one of the central charges of $E_1, E_2$ is zero.  
	
	(ii) A filtration of $E$ as in condition (7) is also called a Harder-Narasimhan filtration of $E$, $Q_i$ is called the $i$-th Harder-Narasimhan factor of such a filtration. Its uniqueness will be discussed in Section \ref{section of uniqueness of HN filtration}. Also note that we do not require  $E_i, Q_i$ to be indecomposable.
	
	(iii) Sometimes, such a stability condition may contain more information than we need. Though these extra information could be encoded in a geometric picture, we still want to have an equivalence relation among such stability conditions. This equivalence relation will also be introduced in Section \ref{section of uniqueness of HN filtration}.
	
	(iv) From the definition, one can easily show that for any integer $i$ and any connecting path $l$  between two indecomposable summands of $Q_i$ in condition (7), we have $q(l)=0$.

	(v) To distinguish two different terms of stability conditions, we will call the notion in Definition \ref{stability condition on cyclic categories} stability conditions, while call the notion in Definition \ref{slicing} Bridgeland stability conditions. Usually, the distinction will be  clear in context, depending whether we assume the triangulated category to be cyclic or not.  In the next section, we will see that these two notions of stability conditions though different in general, are closely related to each other.
	
\end{remarks}

\section{The $\mathbb{Z}$-lifts of cyclic categories and stability conditions}

Let $Stab_{cyc}(\mathcal{C})$ to denote the set of stability conditions on a $k$-linear Krull-Schmidt cyclic category. In this section, we will  investigate some basic properties of $Stab_{cyc}(\mathcal{C})$ and other closely related topological spaces. 

We start with the definition of charge triples and charge pairs.

\begin{definition}
	Assume that $\mathcal{C}$ is a $k$-linear Krull-Schmidt cyclic category, then a charge triple $R=(Z,\phi, q)$  consists of a central charge $Z:K_0(\mathcal{C})\rightarrow\Lambda\rightarrow \mathbb{C}$,  a map $\phi$ sending every indecomposable object $E$ to its phase $\phi(E)\in(0,2]$, and $q$ a degree function of a real decomposition on $\mathcal{C}$, which satisfy conditions (1), (2)  and (3) in Definition \ref{stability condition on cyclic categories}.
	
	A pair $(Z,q)$ is called a charge pair if for any two indecomposable objects $E_1,E_2$ with $Z(E_i)=m(E_i)e^{i\theta_i}\neq 0$ for $i=1,2$, we have that $$q(l)=\theta_2-\theta_1 (mod\ 2\ \mathbb{Z})$$ for any connecting path $l:E_1\dashrightarrow E_2$.
\end{definition}

\begin{remark}
	As we will see Lemma \ref{Triples and pairs}, in most cases a charge triple $(Z,\phi,q)$ is determined by its charge pair $(Z,q)$. 
\end{remark}

Let us fix the set of homogeneous morphisms first. And denote the set of charge triples and charge pairs of such a fixed set of homogeneous morphisms on $\mathcal{C}$ by $\textbf{T}(\mathcal{C})$ and $\textbf{P}(\mathcal{C})$ respectively.  There are  natural topologies on these two sets, which are the coarsest topologies such that the forgetful maps $$\textbf{T}(\mathcal{C})\rightarrow Hom(\Lambda,\mathbb{C})\simeq \mathbb{C}^{rank(\Lambda)}, \  \textbf{P}(\mathcal{C})\rightarrow Hom(\Lambda,\mathbb{C})\simeq \mathbb{C}^{rank(\Lambda)}, $$    $$where \ (Z,\phi,q)\mapsto Z,\ (Z,q)\mapsto Z,$$
 $$\textbf{T}(\mathcal{C})\rightarrow  \mathbb{R},\ \textbf{P}(\mathcal{C})\rightarrow  \mathbb{R},$$   
 $$where \ (Z,\phi,q)\mapsto q(f),\ (Z,q)\mapsto q(f),$$
 and $$\textbf{T}(\mathcal{C})\rightarrow  (0,2]\xrightarrow{e^{i\pi x}} S^1,$$ $$where \ (Z,\phi,q)\mapsto \phi(E)$$  are continuous for any indecomposable object $E$ and any nonzero homogeneous morphism $f$. Here $\mathbb{R}$, $S^1$ and $\mathbb{C}^{rank(\Lambda)}$ are endowed with the standard Euclidean topology. There is also a continuous forgetful map $$\textbf{T}(\mathcal{C})\rightarrow \textbf{P}(\mathcal{C}).$$ 

 \begin{lemma}\label{Triples and pairs}
 	Let $\mathcal{C}$ be a $k$-linear Krull-Schmidt connective cyclic category, the continuous map $$h: \textbf{T}(\mathcal{C})\rightarrow \textbf{P}(\mathcal{C})$$ is an isomorphism except on the locus where the central charge $Z$ is trivial.
 \end{lemma}
 
 \begin{proof}
 	Suppose that there exists an indecomposable object $E$ with $Z(E)\neq 0$, then $\phi(E)$ is determined by $Z(E)$. For any other indecomposable object $F$, by the assumption that $\mathcal{C}$ is connective, there exists a connecting path from $E$ to $F$. Hence $\phi(E)$ and $q(l)$ determines $\phi(F)$.
 \end{proof}
 
 \begin{remark}
 	One can easily show that on the locus where the central charge is trivial, the map $h$ is an $S^1$ bundle map. By condition (5) in Definition \ref{stability condition on cyclic categories}, there is no stability conditions above this locus.
 \end{remark}

  In the next subsection, we will show that on distinguished components in $\textbf{T}(\mathcal{C})$ with vanishing Maslov indices,  the stability conditions can be lifted to be Bridgeland stability conditions.

\subsection{Maslov index and $\mathbb{Z}$ lifting}

Given a real decomposition on $\mathcal{C}$, and assume that $\mathcal{C}$ is pre-liftable with respect to the real decomposition. We can define what is a basic loop in such a real decomposition.
\begin{definition}\label{Fundamental distinguished hexagon}
	Assume that $\mathcal{C}$ admits a real decomposition and $q$ is the associated degree function, a fundamental distinguished hexagon in  $\mathcal{C}$ is the following locally liftable diagram 
	
	$$\begin{tikzcd}
		& \bigoplus_{j=1}^mB_j\arrow{r}{g} & \bigoplus_{k=1}^n C_k\arrow{rd}{h} & \\ \bigoplus_{i=1}^lA_i\arrow{ru}{f} & & & \bigoplus_{i=1}^l A_i[1]\arrow{ld}{f[1]} \\ & \bigoplus_{k=1}^nC_k[1]\arrow{lu}{\beta_A\circ h[1]}  & \bigoplus_{j=1}^mB_j[1] \arrow{l}{g[1]}
	\end{tikzcd}$$
 where $$\bigoplus_{i=1}^lA_i\xrightarrow{f} \bigoplus_{j=1}^mB_j\xrightarrow{g}\bigoplus_{k=1}^nC_k\xrightarrow{h}\bigoplus_{i=1}^lA_i[1]$$ is a distinguished triangle in $\mathcal{C}$.

A basic loop $l$ in $\mathcal{C}$ is $\beta_{A_i}\circ f[1]\circ f$, where $f:A_i\dashrightarrow A_i[1]$ is a connecting path in the diagram of the distinguished triangle.  
\end{definition}
\begin{remark}
	The same loop can also be defined as basing at $B_j$ or $C_k$ as $\beta$ is a natural transformation.
\end{remark}

We can define what is the Maslov index of a basic loop.

\begin{definition}\label{Maslov index}
Given a basic loop in $\mathcal{C}$ as in Definition \refeq{Fundamental distinguished hexagon}, its Maslov index is $$M(l)\coloneqq  \frac{q(f)-1}{2}.$$
\end{definition}

As its namesake in Fukaya categories, we will show that if all the Maslov indices vanish, there is a suitable $\mathbb{Z}$-lift of $\mathcal{C}$.

\begin{theorem}\label{lift of cyclic categoryies}
	Given a stability condition $\sigma=(\mathcal{Q},Z,\phi, q)$ on a $k$-linear Krull-Schmidt cyclic category $\mathcal{C}$, we assume that  the Maslov indices of  all basic loops are zero. There are  $\mathbb{Z}$-lifts of $\sigma$ and $\mathcal{C}$, which we denote by $\widetilde{\sigma}$ and $\widetilde{\mathcal{C}}$ respectively, such that $\widetilde{\sigma}$ is a Bridgeland stability condition on $\widetilde{\mathcal{C}}$.
\end{theorem}
\begin{proof}
	The natural $\mathbb{Z}$-lift of $\mathcal{C}$ is defined in the following way: the indecomposable objects are pairs $(E,\phi)$, where $E$ is an indecomposable object in $\mathcal{C}$ and $\phi\in\mathbb{R}$ such that $$\phi(E)\equiv \phi (mod\ 2\mathbb{Z}).$$ 
	
	The morphism space between two such indecomposable objects is defined by the following formula $$Hom_{\widetilde{\mathcal{C}}}((E_1,\phi_1),(E_2,\phi_2))=\{f\in  Hom_{\mathcal{C}}(E_1,E_2)| f\ is \ homogeneous \ and\ q(f)=\phi_2-\phi_1\}.$$ 

	Moreover, one can easily show that the composition of morphisms is well defined because of the condition (2)  of real decomposition. Note that we can define a $k$-linear functor $\pi:\widetilde{\mathcal{C}}\rightarrow \mathcal{C}$ in a natural way. It sends objects $(E,\phi)$ to $E$ and be the natural inclusion on the spaces of morphisms.  
	
	The triangulated structure on $\widetilde{\mathcal{C}}$ is defined in the following way. The shift functor $[1]$ sends $(E,\phi)$ to $(E[1],\phi+1)$, and is the natural isomorphism on the spaces of morphisms due to condition (1) of real decomposition. $$Hom_{\widetilde{\mathcal{C}}}((E_1,\phi_1),(E_2,\phi_2))\xrightarrow{\sim}Hom_{\widetilde{\mathcal{C}}}((E_1[1],\phi_1+1),(E_2[1],\phi_2+1)).$$  Obviously, one has $\pi\circ  [1]= [1]\circ \pi$.  We define a triangle in $\widetilde{\mathcal{C}}$ to be a distinguished triangle if and only if its image under $\pi$ is a distinguished triangle. 
	
	We need to check that this indeed define a triangulated category $\widetilde{\mathcal{C}}$ (for the definition of triangulated category, we refer to  \cite[Chapter 1]{Neemanstriangulatedcategories}). 
	
	First of all, assume that we have a morphism $$f:\bigoplus_{i=1}^l (A_i,\phi_i)\rightarrow\bigoplus_{i=1}^m (B_j,\psi_j)$$ in $\widetilde{\mathcal{C}}$. The image of $f$ under $\pi$ is a liftable quasi-homogeneous morphism, hence by the assumption that $\mathcal{C}$ is pre-liftable, there exists a distinguished triangle $$\bigoplus_{i=1}^lA_i\xrightarrow{\pi(f)} \bigoplus_{j=1}^mB_j\xrightarrow{g}\bigoplus_{k=1}^nC_k\xrightarrow{h}\bigoplus_{i=1}^lA_i[1]$$ in $\mathcal{C}$, and the diagram is liftable. Therefore, there exists a pre-image of such distinguished triangle in $\widetilde{\mathcal{C}}$ $$\bigoplus_{i=1}^l (A_i,\phi_i)\xrightarrow{f}\bigoplus_{i=1}^m (B_j,\psi_j)\xrightarrow{g}\bigoplus_{1\leq k\leq n, d\in\mathbb{Z}} (C_k, \phi(C_k)+2d)\xrightarrow{h}\bigoplus_{i=1}^l(A[1], \phi_i+1).$$ 
	
	 The last term can be written as $\bigoplus_{i=1}^l(A[1], \phi_i+1)$ since all Maslov indices of basic loops vanish.
	
	Hence axioms \textbf{TR1} is satisfied by our assumption that $\mathcal{C}$ is pre-liftable and all Maslov indices vanish. By Corollary \ref{Fill-in corollary}, we know that $\textbf{TR3}$ is also satisfied. The axiom \textbf{TR2}  hold in $\widetilde{\mathcal{C}}$ since they already hold in $\mathcal{C}$. The octahedral axiom \textbf{TR4} follows from condition (2) in Definition \ref{preliftability}.
	
	For $\widetilde{\sigma}=(\mathcal{P},\widetilde{Z})$ on $\widetilde{\mathcal{C}}$, we define the central charge $\widetilde{Z}$ to be the composition $K_0(\widetilde{\mathcal{C}})\xrightarrow{\pi_0}K_0(\mathcal{C})\xrightarrow{Z} \mathbb{C}$. And $\mathcal{P}(\phi)$ to be the full subcategory consists of objects $(E,\phi)$ such that $E\in\mathcal{Q}(\phi')$ where $\phi'\equiv \phi\ (mod \ 2\mathbb{Z})$. One can easily check that $\widetilde{\sigma}$ is a Bridgeland stability condition on $\widetilde{\mathcal{C}}$ by unwinding the definitions.
\end{proof}

\begin{remarks}\label{Remarks after the lifting theorem}
	(1) The Octahedral Axiom (TR 4) of $\widetilde{\mathcal{C}}$ holds  by the second condition in Definition \ref{preliftability}. However, in practice, this condition could be difficult  to check. The reader could ignore that condition, and take $\widetilde{\mathcal{C}}$ to be just a pre-triangulated category. All the results in this paper holds in that situation, as we do not use the Octahedral Axiom in  our proofs. 
	
(2) We call the functor $\pi:\widetilde{\mathcal{C}}\rightarrow \mathcal{C}$ a $\mathbb{Z}$-covering since it is similar to the $\mathbb{Z}$-covering of topological spaces. We will use $\pi^*(\sigma)$ to denote $\widetilde{\sigma}$. Also note that the $\mathbb{Z}$-lift $\widetilde{\mathcal{C}}$ only depends on the charge triple. By its construction we know that the functor $\pi$ is exact and faithful. 

(3) The functor $\pi$ induce a surjective homomorphism on Grothendieck groups $\pi_0: K_0(\widetilde{\mathcal{C}})\twoheadrightarrow K_0(\mathcal{C})$. Hence the central charge $$Z:K_0(\widetilde{\mathcal{C}})\xrightarrow{\pi_0}K_0(\mathcal{C})\xrightarrow{v}\Lambda\rightarrow\mathbb{C}$$ also factors through the given lattice $\Lambda$. In the following, we usually fix the factorization  $K_0(\widetilde{\mathcal{C}})\xrightarrow{\pi_0}K_0(\mathcal{C})\xrightarrow{v}\Lambda$.
\end{remarks}

According to Theorem \ref{lift of cyclic categoryies}, the following definitions is natural.
\begin{definition}\label{Definition of liftable and deformation equivalence}
	(1) We say that $\mathcal{C}$ is liftable with respect to $R=(Z,\phi,q)$ if $\mathcal{C}$ is pre-liftable with respect to the real decomposition in $R$ and all the Maslov indices vanish.
	
	(2) Let $R_1=(Z_1,\phi_1,q_1)$ and $R_2=(Z_2,\phi_2, q_2)$ be two charge triples on $\mathcal{C}$, if for any two connecting paths $l_1,l_2:E\dashrightarrow F$, we have that $l_1,l_2$ are homogeneous with respect to $R_1$ if and only if $l_1,l_2$ are homogeneous with respect to $R_2$, and moreover $$q_1(l_1)-q_1(l_2)=q_2(l_1)-q_2(l_2).$$ 
	Then we say that $R_1, R_2$ are deformation equivalent.
\end{definition}

\begin{prop}\label{connections of Z-lifts}
	Let $\mathcal{C}$ be a $k$-linear Krull-Schmidt connective cyclic category, and $R_1=(Z_1,\phi_1,q_1)$ and $R_2=(Z_2,\phi_2, q_2)$ be two charge triples on $\mathcal{C}$. Suppose that $R_1,R_2$ are deformation equivalent and $\mathcal{C}$ is  liftable with respect to $R_1$. Then $\mathcal{C}$ is also liftable with respect to $R_2$.  
	
	Moreover,  if we denote the $\mathbb{Z}$-lifts by $\widetilde{\mathcal{C}}_1$ and $\widetilde{\mathcal{C}}_2$ respectively, there exists an  equivalence (not canonical)  $H:\widetilde{\mathcal{C}}_1\simeq \widetilde{\mathcal{C}}_2$.
\end{prop}

\begin{proof}
By definition, it is easy to see  that  $\mathcal{C}$ is pre-liftable with respect to $R_2$. Indeed, for any basic loop $l_1:A\dashrightarrow A$ in $\mathcal{C}$, take $l_2=id_A$, we get $0=q_1(l_1)=q_2(l_2)$. Hence $\mathcal{C}$ is also liftable with respect to $R_2$.
	
	To construct an equivalence $H$, let  $E$ be a nontrivial indecomposable object in $\mathcal{C}$, we can define $H:\widetilde{\mathcal{C}}_1\rightarrow \widetilde{\mathcal{C}}_2$ in the following way. First of all $$H((E,\phi_1(E)+2k))=(E,\phi_2(E)+2k), $$ for any $k\in\mathbb{Z}$. 
	
	Take another indecomposable object $F$, there exists a connecting path $l$ starting from $E$ ending at $F$ in $\widetilde{\mathcal{C}}_1$. We set $$H((F,\phi_1(E)+q_1(l)+2k))\coloneqq (F,\phi_2(E)+q_2(l)+2k)$$ for any $k\in\mathbb{Z}$.  Our assumption $$q_1(l_1)-q_1(l_2)=q_2(l_1)-q_2(l_2) $$ ensures that it is independent of the choice of connecting paths. As $\mathcal{C}$ is connective,  there exists a connecting path $l':F\dashrightarrow F[1]$. As $l''\coloneqq \overline{\beta\circ l'[1]}:F\dasharrow F[1]$ is another connecting from $F$ tp $F[1]$. By assumption we know that $$2q_1(l')=q_1(l')-q_1(l'')=q_2(l')-q_2(l'')=2q_2(l'), $$ which implies that $q_1(l')=q_2(l')$. Hence we get \begin{equation*}
		\begin{split}
		H((F,\phi_1(E)+q_1(l)+q_1(l')+2k-1)[1])  & =H((F[1], \phi_1(E)+q_1(l'\circ l)+2k))\\ &= (F[1],\phi_2(E)+q_2(l')+q_2(l)+2k) \\ &=(F,\phi_2(E)+q_1(l')+q_2(l)+2k-1)[1]
		\end{split}
	\end{equation*} This implies that $H\circ [1]=[1]\circ H$ on objects.
	
	For the morphisms, assume that $$f\in Hom_{\mathcal{C}}(G, G')$$ is a homogeneous morphism between  two indecomposable objects $G,G'\in\mathcal{C}$, and $$l:E\dashrightarrow G$$ is a connecting path from $E$ to $G$.  The morphism $f$ can be lifted to $$f'\in Hom_{\widetilde{\mathcal{C}}_1}((G,\phi_1(E)+q_1(l)+2k),(G', \phi_1(E)+q_1(l)+q_1(f)+2k)),$$ and we define $$H(f')\in Hom_{\widetilde{\mathcal{C}}_2}((G,\phi_2(E)+q_1(l)+2k), (G',\phi_2(E)+q_2(l)+q_2(f)+2k)).$$

	From the definition of distinguished triangles in $\widetilde{\mathcal{C}}_1$ and $\widetilde{\mathcal{C}}_2$, it is easy to see that $H:\widetilde{\mathcal{C}}_1\rightarrow \widetilde{\mathcal{C}}_2$ sends distinguished triangles to distinguished triangles. Therefore, $H$ is an exact functor between triangulated categories.

	By the symmetry of $R_1$ and $R_2$, we can define a functor $G:\widetilde{\mathcal{C}}_2\rightarrow\widetilde{\mathcal{C}}_1$ starting from the same indecomposable object $E\in\mathcal{C}$.
		It is easy to check that $H,G$ are inverse to each other as exact functors between triangulated categories.
\end{proof}
\begin{remarks}\label{Remarks below Z-lifts}

(i) As we can see in the proof, different choices of $E$ may result in different equivalences, hence the equivalence is not canonical. By the assumption that $\mathcal{C}$ is connective, one can show that there are $\mathbb{Z}$-copies of equivalences in total.

(ii) This equivalence induces a commutative diagram between coverings.

$$\begin{tikzcd}
	\widetilde{\mathcal{C}}_1 \arrow{rd}[swap]{\pi_1} \arrow{rr}{H}& & \widetilde{\mathcal{C}}_2 \arrow {ld}{\pi_2} \\ & \mathcal{C} & 
\end{tikzcd}$$

(iii) Consistent choices of such equivalences can be viewed as a connection on the $\mathbb{Z}$-lifts fibered over a connected component of $\textbf{T}(\mathcal{C})$ where $\mathcal{C}$ is liftable with respect to the charge triples in that component.

(iv) In theory, there could be multiple components of charge triples on which $\mathcal{C}$ is liftable, and the lifted categories are not equivalent to each other.

\end{remarks}

\section{Uniqueness of Harder-Narasimhan filtration}\label{section of uniqueness of HN filtration}

In this section, we will discuss the uniqueness of Harder-Narasimhan filtration. We start with the following lemma.



\begin{lemma}\label{important lemma}
	Let $\mathcal{C}$ be a $k$-linear Krull-Schmidt connective cyclic category,  $\sigma=(\mathcal{Q}, Z,\phi,q)$ be a stability condition on $\mathcal{C}$. Assume that all the Maslov indices of basic loops are nonnegative and let 
		$$\begin{tikzcd}[column sep=large]
		& \cdots\arrow{d} & \cdots\arrow{d} & \cdots & \cdots \arrow{d}& \cdots\arrow{d} \\
		& E_0 \arrow{d}{f_1}&E_1\arrow{d}{f_2} & \cdots & E_{n-2} \arrow{d}{f_{n-1}} &  E_{n-1}\arrow{d}{f_n}\\ 
		0=E_0  \arrow{ru}{id} & E_1\arrow{d}{p_1} \arrow{ru}{id} & E_2\arrow{d}{p_2}\arrow{ru}{id} &\cdots  \arrow{ru}{id}& E_{n-1}\arrow{ru}{id} \arrow{d}{p_{n-1}} &E_n=E\arrow{d}{p_n}\\ & Q_1 \arrow{d}& Q_2\arrow{d}& \cdots & Q_{n-1}\arrow{d} & Q_n\arrow{d}\\ & \cdots & \cdots & \cdots & \cdots & \cdots 
	\end{tikzcd} $$
	be a Harder-Narasimhan filtration of an indecomposable object $E$ as in condition (7) of Definition \ref{stability condition on cyclic categories} with $n\geq 2$. Further assume that we have the following liftable commutative diagram between distinguished triangles,
$$\begin{tikzcd}
	\cdots \arrow{r} &E_{i-1} \arrow{r}{f_i} \arrow{d}[swap]{\alpha}& E_i \arrow{r}{p_i}\arrow{d}{id} & Q_i \arrow{d}{\delta} \arrow{r}{t_i} & E_{i-1}[1] \arrow{d}{\alpha[1]}\arrow{r} & \cdots \\ \cdots \arrow{r} &	E_{i-1} \arrow{r}{f_i} & E_i \arrow{r}{p_i} & Q_i \arrow{r}{t_i}  & E_{i-1}[1] \arrow{r} & \cdots 
\end{tikzcd}$$ and if we glue two copies of Harder-Narasimhan diagrams along each rows in this diagram respectively, we get a liftable commutative diagram. Then $\delta=id$ and  $\alpha=id$.
\end{lemma}

\begin{proof}

	Firstly, we claim that: there is a connecting path $l:E_{n-1,1}\dashrightarrow E_{n-1,1}[1]$  in the following distinguished triangle $$E_{n-1}\xrightarrow{f_n} E_n\xrightarrow{p_n} Q_n\xrightarrow{t_n}E_{n-1}[1],$$ where  $E_{n-1,1}$ is an indecomposable summand of $E_{n-1}$.
	
	Let $E_{n-1,1}$ be an arbitrary indecomposable summand of $E_{n-1}$, if we assume that $p[1]\circ t_n=0$ where $p:E_{n-1}\rightarrow E_{n-1,1}$ is the natural projection map.  We have the following diagram $$\begin{tikzcd}
		E_{n-1}\arrow{r}{f_n} \arrow{d}{p} & E_n \arrow{r}{p_n}\arrow{d}{f} & Q_n\arrow{r}{t_n}\arrow{d}{id} & E_{n-1}[1] \arrow{d}{p[1]} \\ E_{n-1,1}\arrow{r} & Q_n\oplus E_{n-1,1}\arrow{r} & Q_n \arrow{r}{0} & E_{n-1,1}[1]
	\end{tikzcd}$$ If we can precompse the natural inclusion $i:E_{n-1,1} \rightarrow E_{n-1}$ and postcompose the natural projection $q:Q_n\oplus E_{n-1,1}\rightarrow E_{n-1,1}$ with the diagram, we can see that $E_{n-1,1}$ is a direct summand of $E_n$. As $E_n$ is indecomposable, we get $E_{n-1,1}\simeq E_n$. Then one can easily show that there is a morphism $i': E_n\rightarrow E_{n-1}$ such that $f_n\circ i'=id$, this implies that $p_n=0$, which contradicts the first condition in Definition \ref{stability condition on cyclic categories}.(7). Hence we showed that $p\circ t_n\neq 0$ for any indecomposable direct summand $E_{n-1,1}$ of $E_{n-1}$. 

This easily implies the claim. Indeed, by the connectivity condition in Definition \ref{stability condition on cyclic categories}.(7), there is an indecomposable summand $E_{n-1,1}$ with $f_n|_{E_{n-1,1}}\neq 0$, and any summand $p'$  of $p_n$ is nonzero. Hence we get a connecting loop $l:E_{n-1,1}\dashrightarrow E_{n-1,1}[1]$. By the assumption that all Maslov indices are nonnegative, we have $q(l)\geq 1$.

		We use $\textbf{D}$ to denote the  glued diagram of two Harder-Narasimhan diagrams.  As in Remark \ref{Remarks under the definition of stability conditions}.(iv), it is easy to see that every summand of $\delta$ is a homogeneous morphism of degree $0$. Hence, if we replace $\delta$ by $\delta-id$ in $\textbf{D}$, the new diagram is still liftable.

		By assumption, we get that $(id-\delta)\circ p_i=0$. Hence $id-\delta$ is in the image of the map $$Hom_{\mathcal{C}}(E_{i-1}[1],Q_i)\xrightarrow{\circ t_i} Hom_{\mathcal{C}}(Q_i,Q_i),$$ i.e. we can write $id-\delta=\delta_i\circ t_i$, where $\delta_i$ can be chosen to be a quasi-homogeneous morphism such that the following diagram is commutative and liftable by Corollary \ref{Fill-in corollary}. $$\begin{tikzcd} 
	\cdots\arrow{r} & Q_i \arrow{d}{id-\delta} \arrow{r}{t_i} & E_{i-1}[1] \arrow{ld}{\delta_i} \arrow{r} & \cdots \\  & Q_i & &
	\end{tikzcd}$$ where the first row is the Harder-Narasimhan diagram of $E$.

	We assume that $id-\delta\neq 0$, hence $t_i\neq 0, \delta_i\neq 0$. We have the following liftable commutative diagram $$\begin{tikzcd}
 E_{i-2}[1]\arrow{r}{f_{i-1}[1]}  & E_{i-1}[1] \arrow{d}{\delta_i} \arrow{r}{p_{i-1}[1]} & Q_{i-1}[1]  \\ &  Q_i & 
	\end{tikzcd}$$ 

Our second claim is that $\delta_i$ is not in the image of $$Hom_{\mathcal{C}}(Q_{i-1}[1],Q_i)\xrightarrow{\circ p_{i-1}[1]} Hom_{\mathcal{C}}(E_{i-1}[1],Q_i).$$ 

Suppose the contrary, i.e. there exists a nonzero morphism $\epsilon\in Hom_{\mathcal{C}}(Q_{i-1}[1],Q_i)$ with $\delta_i=\epsilon\circ p_{i-1}[1]$. By Corollary \ref{Fill-in corollary}, we can choose $\epsilon$ such that the following diagram is commutative and liftable.$$\begin{tikzcd}
 E_{i-2}[1]\arrow{r}{f_{i-1}[1]}  & E_{i-1}[1] \arrow{d}{\delta_i} \arrow{r}{p_{i-1}[1]} & Q_{i-1}[1] \arrow{ld}{\epsilon}  \\ &  Q_i & 
\end{tikzcd}$$

Let $Q_{i-1,1}[1], Q_{i,1}$ be indecomposable summands of $Q_{i-1}[1], Q_i$ respectively such that the homogeneous summand $\epsilon_{11}:Q_{i-1,1}[1]\rightarrow Q_{i,1}$ is nonzero. By the proof of Lemma \ref{Factorization lemma}, this implies that there is an indecomposable summand $E_{i-1,1}[1]$ with a connecting path $$l_1:Q_{i-1,1}[1]\dashrightarrow E_{i-1,1}[1]$$ in the liftable morphism $E_{i-1}[1]\xrightarrow{p_{i-1}[1]} Q_{i-1}[1]$, and a connecting path $$l_2: E_{i-1,1}[1]\dashrightarrow Q_{i,1}$$ in the Harder-Narasimhan diagram of $E$. 

By the connectivity condition in Definition \ref{stability condition on cyclic categories}.(7), one can find a connecting path $$l_3:E_{i-1,1}\dashrightarrow Q_{i,1}$$ in  the Harder-Narasimhan diagram of $E$. 

By the last condition of Definition  \ref{stability condition on cyclic categories}.(7), we know that \begin{equation}
	0>c_i=q(l_3\circ l_1[-1])=q(l_3)+q(l_1).
\end{equation}

Once again by the connectivity assumption, one can find a connecting path $$l_4:E_{i-1,1}\dashrightarrow E_{n,1}$$ in the Harder-Narasimhan diagram of $E$. Then we have the following  connecting loop.  $$\begin{tikzcd}
	E_{i-1,1}\arrow[r, dotted, "\bar{l}_2\circ l_3"] \arrow[d, dotted, "l_4"]& E_{i-1,1}[1]\arrow[d, dotted, "l_4{[1]}"]  \\ E_{n,1} \arrow[r,dotted, "l"]& E_{n,1}[1]
\end{tikzcd}$$
As the Harder-Narasimhan diagram of $E$ is liftable, we know that \begin{equation}
	-q(l_2)+q(l_3)=q(\bar{l}_2\circ l_3)=q(l)\geq 1>0.
\end{equation} The inequalities (1) and (2) implies that $$q(\epsilon_{11})=q(l_1)+q(l_2)<0, $$ this implies that $\epsilon_{11}=0$ by Definition \ref{stability condition on cyclic categories}, which contradicts the assumption $\epsilon_{11}\neq 0$. Hence our second claim is proved. Therefore, we get that $$\delta_i\circ f_{i-1}[1]\neq 0.$$ 

We use the same argument as in the proof of our second claim to do induction, in the end, we will get that $$\delta_i\circ f_{i-1}[1]\circ \cdots \circ f_2[1]\circ f_1[1]\neq 0, $$ which is impossible since $f_1=0$. Hence we proved that $\delta-id=0$.

The proof for $\alpha=id$ is similar.
\end{proof}

\begin{prop}\label{Nontrivial composition}
Let $\mathcal{C}$ be a $k$-linear Krull-Schmidt connective cyclic category,  $\sigma=(\mathcal{Q}, Z,\phi,q)$ be a stability condition on $\mathcal{C}$. Assume that all the Maslov indices of basic loops are nonnegative and let 	$$\begin{tikzcd}[column sep=large]
	& \cdots\arrow{d} & \cdots\arrow{d} & \cdots & \cdots \arrow{d}& \cdots\arrow{d}  \\ & E_0 \arrow{d}{f_1}&E_1\arrow{d}{f_2} & \cdots & E_{n-2} \arrow{d}{f_{n-1}} &  E_{n-1}\arrow{d}{f_n}\\ 
	0=E_0  \arrow{ru}{id} & E_1\arrow{d}{p_1} \arrow{ru}{id} & E_2\arrow{d}{p_2}\arrow{ru}{id} &\cdots  \arrow{ru}{id}& E_{n-1}\arrow{ru}{id} \arrow{d}{p_{n-1}} &E_n=E\arrow{d}{p_n}\\ & Q_1 \arrow{d}& Q_2\arrow{d}& \cdots & Q_{n-1} \arrow{d}& Q_n \arrow{d} \\ & \cdots & \cdots & \cdots & \cdots & \cdots
\end{tikzcd} $$ be a Harder-Narasimhan filtration of an indecomposable object $E$ as in condition (7) of Definition \ref{stability condition on cyclic categories}. 

Then the compositions $f_{j}\circ f_{j-1}\circ\cdots \circ f_i\neq 0$ for any $2\leq i<j\leq n$.	
\end{prop}
\begin{proof}
	Without loss of generality, we can let $j=n$. By induction, we can assume that $f_{n}\circ f_{n-1}\circ\cdots \circ f_{i+1}\neq 0$, and it suffices to prove that $f_{n}\circ f_{n-1}\circ\cdots \circ f_i\neq 0$. Assume the contrary, i.e. the composition $f_{n}\circ f_{n-1}\circ\cdots \circ f_i$ is trivial.
	
	Since $f_{n}\circ f_{n-1}\circ\cdots \circ f_i=0$, we have the following liftable commutative diagram. 
	$$\begin{tikzcd}[column sep=large]
	& \cdots\arrow{d}  & \cdots & \cdots \arrow{d}& \cdots\arrow{d} \\	& E_{i-1} \arrow{d}{f_{i}}& \cdots & E_{n-2} \arrow{d}{f_{n-1}} &  E_{n-1}\arrow{d}{f_n}\\ 
		E_{i-1} \arrow{ru}{id} & E_i\arrow{d}{p_i} \arrow{ru}{id} &\cdots  \arrow{ru}{id}& E_{n-1}\arrow{ru}{id}  &E_n=E\arrow{d}{p_n}\\ & Q_i \arrow{d}\arrow{rrru}[swap]{h_n} & \cdots & \cdots  & Q_n \arrow{d} \\ & \cdots  & \cdots & \cdots & \cdots
	\end{tikzcd} $$
	
 One can easily see that $p_n\circ h_n=0$ by Definition \ref{stability condition on cyclic categories}. Hence $h_n$ factors through $E_{n-1}$,  we have following liftable diagram 	$$\begin{tikzcd}[column sep=large]
 	& \cdots\arrow{d}  & \cdots & \cdots \arrow{d}& \cdots\arrow{d} \\& E_{i-1} \arrow{d}{f_{i}}& \cdots & E_{n-2} \arrow{d}{f_{n-1}} &  E_{n-1}\arrow{d}{f_n}\\ 
 	E_{i-1} \arrow{ru}{id} & E_i\arrow{d}{p_i} \arrow{ru}{id} &\cdots  \arrow{ru}{id}& E_{n-1}\arrow{ru}{id}  &E_n=E\arrow{d}{p_n}\\ & Q_i \arrow{rrru}[swap]{h_n} \arrow{rru}{h_{n-1}}\arrow{d}& \cdots & \cdots  & Q_n\arrow{d} \\ & \cdots  & \cdots & \cdots & \cdots
 \end{tikzcd} $$
 
 We claim that this diagram  is commutative. It suffices to prove that $$h_{n-1}\circ p_i=f_{n-1}\circ\cdots \circ f_{i+1}.$$ Denote the difference $h_{n-1}\circ p_i-f_{n-1}\circ\cdots \circ f_{i+1}$ by $\delta$, we have $f_n\circ \delta=0$. Hence we have the following liftable commutative diagram $$\begin{tikzcd}
\cdots \arrow{r} &E_{{i-1}}\arrow{r}{f_i} & E_i \arrow{d}{\delta}\arrow{ld}{\epsilon}\arrow{r}{p_i} & Q_i \arrow{r} & \cdots \\  \cdots \arrow{r} & Q_n[-1]\arrow{r}{t} & E_{n-1} \arrow{r}{f_n} & E_n \arrow{r} & \cdots 
\end{tikzcd}$$ where both rows are in the Harder-Narasimhan diagram of $E$.	 By the same argument in the proof of Lemma \ref{important lemma}, we can show that $\delta=0$. Hence, the claim is proved.

 Continuing the same argument, we get two liftable commutative diagrams $$\begin{tikzcd}[column sep=large]
 	 	& \cdots\arrow{d}  & \cdots \arrow{d}& \cdots \arrow{d}& \cdots\arrow{d} \\ & E_{i-1} \arrow{d}{f_{i}}& E_i \arrow{d}{f_{i+1}} & \cdots &  E_{n-1}\arrow{d}{f_n}\\ 
 	E_{i-1} \arrow{ru}{id} & E_i\arrow{d}{p_i} \arrow{ru}{id} & E_{i+1} & \cdots  \arrow{ru}{id}& E_n=E\arrow{d}{p_n}\\ & Q_i \arrow{d} \arrow{ru}{h_{i+1}} \arrow{rrru}[swap]{h_n} & \cdots & \cdots  & Q_n\arrow{d} \\ & \cdots  & \cdots & \cdots & \cdots
 \end{tikzcd} $$ and 
$$\begin{tikzcd}[column sep=large]
& \cdots\arrow{d}  & \cdots \arrow{d}& \cdots \arrow{d}& \cdots\arrow{d} \\ 	& E_{i-1} \arrow{d}{f_{i}}& E_i \arrow{d}{f_{i+1}} & \cdots &  E_{n-1}\arrow{d}{f_n}\\ 
	E_{i-1} \arrow{ru}{id} & E_i \arrow{ru}{id} & E_{i+1} & \cdots  \arrow{ru}{id}& E_n=E\arrow{d}{p_n}\\ & Q_i \arrow{d}\arrow{ru}{h_{i+1}} \arrow{rrru}[swap]{h_n} \arrow{u}{h_i}& \cdots & \cdots  & Q_n\arrow{d} \\ & \cdots  & \cdots & \cdots & \cdots
\end{tikzcd} $$
Hence we get $f_{i+1}\circ (h_i\circ p_i-id)=0$. By the same argument in the proof of Lemma \ref{important lemma}, one can show that  $h_i\circ p_i=id$. Hence $E_i$ is a direct summand of $Q_i$, which implies that $f_i=0$. This contradicts the connectivity condition. Hence, the proposition is proved.
\end{proof}

\begin{theorem}\label{uniqueness of Hn filtrations}
	Let $\mathcal{C}$ be a $k$-linear Krull-Schmidt connective cyclic category,  $\sigma=(\mathcal{Q}, Z,\phi,q)$ be a stability condition on $\mathcal{C}$. Assume that all the Maslov indices of basic loops are nonnegative.
	
	Then the Harder-Narasimhan filtration of any indecomposable object $E$ is unique up an isomorphism.
\end{theorem}

\begin{proof}

	Let the following liftable diagram
		$$\begin{tikzcd}[column sep=large]
		& \cdots\arrow{d} & \cdots\arrow{d} & \cdots & \cdots \arrow{d}& \cdots\arrow{d}  \\ & E_0 \arrow{d}{f_1}&E_1\arrow{d}{f_2} & \cdots & E_{n-2} \arrow{d}{f_{n-1}} &  E_{n-1}\arrow{d}{f_n}\\ 
		0=E_0  \arrow{ru}{id} & E_1\arrow{d}{p_1} \arrow{ru}{id} & E_2\arrow{d}{p_2}\arrow{ru}{id} &\cdots  \arrow{ru}{id}& E_{n-1}\arrow{ru}{id} \arrow{d}{p_{n-1}} &E_n=E\arrow{d}{p_n}\\ & Q_1 \arrow{d}& Q_2\arrow{d}& \cdots & Q_{n-1} \arrow{d}& Q_n \arrow{d} \\ & \cdots & \cdots & \cdots & \cdots & \cdots
	\end{tikzcd} $$ 
be a Harder-Narasimhan diagram of an indecomposable object $E$. By Definition \ref{stability condition on cyclic categories}, we know that for any indecomposable summand $Q_{i,1}$ in $Q_i$, there is a connecting path $l:Q_{i,1}\dashrightarrow E$ in the diagram. By Definition \ref{stability condition on cyclic categories}.(7), $q(l)$ is independent of the path, only depends on the integer $i$. Hence, we can denote this degree by $\phi_i$. Moreover, we also have $$\phi_1<\phi_2<\cdots <\phi_n$$ by definition.

Suppose we have another liftable diagram 	$$\begin{tikzcd}[column sep=large]
	& \cdots\arrow{d} & \cdots\arrow{d} & \cdots & \cdots \arrow{d}& \cdots\arrow{d}  \\ & E_0 \arrow{d}{f_1'}&E_1'\arrow{d}{f_2'} & \cdots & E_{m-2}' \arrow{d}{f_{m-1}'} &  E_{m-1}'\arrow{d}{f_m'}\\ 
	0=E_0  \arrow{ru}{id} & E_1'\arrow{d}{p_1'} \arrow{ru}{id} & E_2'\arrow{d}{p_2'}\arrow{ru}{id} &\cdots  \arrow{ru}{id}& E_{m-1}'\arrow{ru}{id} \arrow{d}{p_{m-1}'} &E_m=E\arrow{d}{p_m'}\\ & Q_1' \arrow{d}& Q_2'\arrow{d}& \cdots & Q_{m-1}' \arrow{d}& Q_m' \arrow{d} \\ & \cdots & \cdots & \cdots & \cdots & \cdots
\end{tikzcd} $$  which is another Harder-Narasimhan diagram of an indecomposable object $E$. Similarly we have a increasing sequence of real numbers $$\phi_1'>\phi_2'>\cdots >\phi_m',$$ where $\phi_i'$ denotes the degree of a connecting path from an indecomposable summand $Q_{i,1}'$ in $Q_i'$ to $E$ in this diagram.

As $E$ is an indecomposable object, we have following liftable diagram. $$\begin{tikzcd}
	\cdots \arrow{r} & E_{n-1}\arrow{r}{f_n} & E\arrow{d}{id} \arrow{r}{p_n} &Q_n \arrow{r} & \cdots  \\	\cdots \arrow{r} & E_{m-1}'\arrow{r}{f_m} & E \arrow{r}{p_m'} &Q_m' \arrow{r} & \cdots  
\end{tikzcd}$$ where the first row is in the first Harder-Narasimhan diagram, and the second row is in the second Harder-Narasimhan diagram.

We claim that $\phi_n=\phi_m'$. To prove the claim, let us assume that $\phi_n<\phi_m'$ first.  Under such assumption, one can show that $p_m'\circ f_n=0$. Indeed, if $p_m'\circ f_n\neq 0$, by the proof in Proposition \ref{Nontrivial composition} and the assumption that $\phi_n<\phi_m'$, we conclude that $$p_m'\circ f_n\circ f_{n-1}\circ \cdots \circ f_1\neq 0,$$ which is a contradiction. Hence we get $p_m'\circ f_n=0$. Therefore, we have a liftable commutative diagram $$\begin{tikzcd}
	\cdots \arrow{r} & E_{n-1}\arrow{r}{f_n} & E\arrow{d}{id} \arrow{r}{p_n} &Q_n \arrow{d}{t} \arrow{r} & \cdots  \\	\cdots \arrow{r} & E_{m-1}'\arrow{r}{f_m} & E \arrow{r}{p_m'} &Q_m' \arrow{r} & \cdots  
\end{tikzcd}$$

However, by the assumption $\phi_n<\phi_m'$, we know that any homogeneous summand of $t$ is of negative degree, hence $t=0$, which implies $p_m'=0$. This contradicts the connectivity condition.  Therefore, we proved $\phi_n\geq \phi_m'$. And by a symmetric argument, we get $\phi_n=\phi_m'$. This proves the claim.

In the case $\phi_n=\phi_m'$, the proof of Lemma \ref{important lemma} could show that $p_n\circ f_m=0$ and $p_m'\circ f_n=0$. Hence we have the following liftable and commutative diagram.

$$\begin{tikzcd}
	\cdots \arrow{r} & E_{n-1}\arrow{r}{f_n} \arrow{d}{\alpha} & E\arrow{d}{id} \arrow{r}{p_n} &Q_n \arrow{d}{t} \arrow{r} & \cdots  \\	\cdots \arrow{r} & E_{m-1}'\arrow{r}{f_m} \arrow{d}{\alpha'}& E \arrow{r}{p_m'} \arrow{d}{id}&Q_m' \arrow{r} \arrow{d}{t'}& \cdots  \\	\cdots \arrow{r} & E_{n-1}\arrow{r}{f_n} \arrow{d}{\alpha} & E\arrow{d}{id} \arrow{r}{p_n} &Q_n \arrow{d}{t} \arrow{r} & \cdots  \\ \cdots \arrow{r} & E_{m-1}' \arrow{r}{f_m} &E \arrow{r}{p_m'} & Q_m' \arrow{r} & \cdots
\end{tikzcd}$$
By Lemma \ref{important lemma}, we get that $Q_n\simeq Q_m'$ and $E_{n-1}\simeq E_{m-1}'$. 

Inductively using the similar argument, one can conclude that $n=m$, $\phi_i=\phi_i'$, $E_i\simeq E_i'$, $Q_i\simeq Q_i'$, and the last two class of isomorphisms are compatible. Hence these two Harder-Narasimhan filtrations are isomorphic.

\end{proof}

This theorem leads to the following natural definition of stability conditions.

\begin{definition}
A stability condition $\sigma$ is called a good stability condition if the Harder-Narasimhan filtration of any indecomposable object is unique. We denote such set by $Stab_{cyc}^u(\mathcal{C})$. 
\end{definition}
\begin{remark}
	By Theorem \ref{uniqueness of Hn filtrations}, we know that $Stab_{cyc}^u(\mathcal{C})$ differs from $Stab_{cyc}(\mathcal{C})$ only on the locus where there are some basic loops with negative Maslov indices.
	
	Also note that, although the non-negativity pf Maslov indices is a sufficient condition for the uniqueness of Harder-Narasimhan filtration, it is not a necessary condition. 
\end{remark}
As mentioned in Remark \ref{Remarks under the definition of stability conditions}.(iii), the space of good stability condition as in Definition \ref{stability condition on cyclic categories} may contain more information than we need. Therefore, we have following equivalence relation between good stability conditions. 
\begin{definition}\label{equivalent relations}
	Let $\sigma_1=(\mathcal{Q}_1, Z_1, \phi_1, q_1)$ and $\sigma_2=(\mathcal{Q}_2, Z_2, \phi_2, q_2)$ be two good stability conditions on $\mathcal{C}$, we say that $\sigma_1$ is equivalent to $\sigma_2$ if the following conditions are satisfied:
	
	\begin{enumerate}
		\item The circle slicings and central charges are the same, i.e. $\mathcal{Q}_1=\mathcal{Q}_2$ and $Z_1=Z_2$.
		\item A morphism $f$ is homogeneous with respect to $\sigma_1$ if and only if it is homogeneous with respect to $\sigma_2$.
		\item For any indecomposable objects $E$, the HN filtration of $E$ with respect to $\sigma_1$  is isomorphic to the HN filtration of $E$ with respect to $\sigma_2$.
\item For any connecting path $l$ from a semi-stable object to another semi-stable object, we have $q_1(l)=q_2(l)$.
	\end{enumerate}

\end{definition}

\begin{remark}
		It is obviously an equivalence relation. We use $\sigma_1\simeq \sigma_2$ to denote such a relation.
	
	We denote the equivalent classes of good stability conditions by $Stab_{cyc}^{u,e}(\mathcal{C})$. We have the following diagram. $$Stab_{cyc}^{u,e}(\mathcal{C})\twoheadleftarrow Stab_{cyc}^{u}(\mathcal{C})\hookrightarrow Stab_{cyc}(\mathcal{C}).$$

\end{remark}
The following two results are the main reasons of introducing this equivalence relation. The notations are the same as in Proposition \ref{connections of Z-lifts} and Remarks \ref{Remarks below Z-lifts}.

\begin{prop}\label{equivalence proposition}
	Let $\mathcal{C}$ be a $k$-linear Krull-Schmidt connective cyclic category, and $\sigma_1,\sigma_2$ be two equivalent good stability conditions on $\mathcal{C}$. Then the charge triples $R_1,R_2$ of $\sigma_1$ and $\sigma_2$ are deformation equivalent. In particular, the Maslov indices of a basic loop with respect to $\sigma_1,\sigma_2$ are the same. 
	
	Moreover, if $\mathcal{C}$ is liftable with respect to $\sigma_1$ (hence also liftable with respect to $\sigma_2$). Then there exists a canonical equivalence $H:\widetilde{\mathcal{C}}_1\rightarrow\widetilde{\mathcal{C}}_2$ between two $\mathbb{Z}$-coverings of $\mathcal{C}$, such that $$H^*\pi_2^*\sigma_2=\pi_1^*\sigma_1,$$ where $\pi_1,\pi_2$ are projections in the following commutative diagram. $$\begin{tikzcd}
		\widetilde{\mathcal{C}}_1 \arrow{rd}[swap]{\pi_1} \arrow{rr}{H}& & \widetilde{\mathcal{C}}_2 \arrow {ld}{\pi_2} \\ & \mathcal{C} & 
	\end{tikzcd}$$ 
\end{prop}

\begin{proof}
To prove that $R_1,R_2$ are deformation equivalent, let $l$ be a connecting path from $E$ to $F$, and consider the Harder-Narasimhan filtrations of $E$ and $F$ respectively. There are connecting paths $l_1:Q_{1,1}\dashrightarrow E$ and $l_2:F\dashrightarrow Q_{m,1}'$, where $Q_{1,1}$ is an indecomposable summand in the first Harder-Narasimhan factor of $E$ and $Q_{m,1}'$ is an indecomposable summand in the last Harder-Narasimhan factor of $F$. Therefore, we get a connecting path $$l_2\circ l\circ l_1:Q_{1,1}\dashrightarrow Q_{m,1}'.$$

By Definition \ref{equivalent relations}.(4), we get $$q_1(l_2\circ l\circ l_1)=q_2(l_2\circ l\circ l_1), $$ which implies that $q_1(l)-q_2(l)$ is a constant real number for any connecting path $l:E\dashrightarrow F$. This proves that $R_1,R_2$ are deformation equivalent.

For the second half of the proposition, we assume that $\mathcal{C}$ is liftable with respect to $\sigma_1$. Our equivalence $H$ is constructed in the following way. 

We start with a semi-stable object $E\in\mathcal{Q}(\phi)$, then proceed as in the proof of Proposition \ref{connections of Z-lifts}. One can easily show that $$H^*\pi_2^*\sigma_2=\pi_1^*\sigma_1$$ just by unwinding the definitions. This equivalence is independent of the choice of semi-stable object $E$ because of Definition \ref{equivalent relations}.(5).
\end{proof}
We have the following theorem relating the stability conditions on $\mathcal{C}$ to Bridgeland stability conditions on its $\mathbb{Z}$-lift. In this theorem, we fix  the factorization  $K_0(\widetilde{\mathcal{C}})\xrightarrow{\pi_0}K_0(\mathcal{C})\xrightarrow{v}\Lambda$ as in Remarks \ref{Remarks after the lifting theorem}.(3).
\begin{theorem}\label{Main theorem}
	Let $\sigma$ be a stability condition on a connective cyclic category $\mathcal{C}$ such that $\mathcal{C}$ is liftable with respect to $\sigma$. By Theorem \ref{lift of cyclic categoryies}, we have a $\mathbb{Z}$-lift $\widetilde{\mathcal{C}}$ of the cyclic category $\mathcal{C}$. We use $Stab_{cyc}^0(\mathcal{C})$ to denote the set of stability conditions whose charge triples are deformation equivalent to the charge triple of $\sigma$, and $Stab_{cyc}^{0,e}(\mathcal{C})$ to denote the equivalent classes $Stab_{cyc}^0(\mathcal{C})/\simeq $.
	
	Then there is an isomorphism $$Stab(\widetilde{\mathcal{C}})/2\mathbb{Z}\xrightarrow{\sim} Stab_{cyc}^{0,e}(\mathcal{C}), $$ where $Stab(\widetilde{\mathcal{C}})$ denotes the space of Bridgeland stability conditions on $\widetilde{\mathcal{C}}$, and the $2\mathbb{Z}$ action is given by $2k\mapsto [2k]$ for any $k\in\mathbb{Z}$. 
\end{theorem}
\begin{proof}
	Let us first give a map $$\mathbb{L}:Stab^{0,e}_{cyc}(\mathcal{C})\rightarrow Stab(\widetilde{\mathcal{C}})/2\mathbb{Z}.$$

	Fix an indecomposable object $G\in\mathcal{C}$, for any stability condition $\sigma_1\in Stab^0(\mathcal{C})$, we have an equivalence $$H:\widetilde{\mathcal{C}}\rightarrow \widetilde{\mathcal{C}}_1$$ by Proposition \ref{connections of Z-lifts}, where $\widetilde{\mathcal{C}}_1$ is the $\mathbb{Z}$-lift of $\mathcal{C}$ with respect to $\sigma_1$. We define $$\mathbb{L}(\sigma_1)=H^*\pi_1^*\sigma_1,$$  where $\pi_1:\widetilde{\mathcal{C}}_1\rightarrow \mathcal{C}$ is the natural covering functor. This map $\mathbb{L}$ is well defined on equivalence classes by Proposition \ref{equivalence proposition} and the fact that these equivalent functors  form a connection (as we fix an indecomposable object $G$, see Remark \ref{Remarks below Z-lifts}). 
	
	On the other hand, we need to  construct a map $$\mathbb{P}:Stab(\widetilde{\mathcal{C}})/2\mathbb{Z}\xrightarrow{} Stab_{cyc}^{0,e}(\mathcal{C}).$$  Given a Bridgeland stability condition $\widetilde{\sigma}_1=(\mathcal{P},\widetilde{Z})$ on $\widetilde{\mathcal{C}}$, we can define a stability condition $\mathbb{P}(\widetilde{\sigma}_1)\coloneqq \sigma_1=(\mathcal{Q},Z,\phi,q)$ on $\mathcal{C}$ in the following way. 
	
	For the circle slicing $\mathcal{Q}$, let $\mathcal{Q}(\phi)$ be the full subcategory consisting of objects $\pi(\mathcal{P}(\phi))$ for any $\phi\in (0,2]$.  
	
For the central charge, we know that $\widetilde{Z}$ factors as $$\widetilde{Z}:K_0(\widetilde{\mathcal{C}})\xrightarrow{\pi_0}K_0(\mathcal{C})\xrightarrow{v}\Lambda\xrightarrow{g}\mathbb{C}, $$  we let $Z:K_0(\mathcal{C})\xrightarrow{v}\Lambda\xrightarrow{g}\mathbb{C}$ be the central charge in $\sigma_1$.
	 
	For the phase function $\phi$, let $E$ be an indecomposable object in $\mathcal{C}$, if $$Z(E)=m(E)e^{i\pi \phi}\neq 0$$ and $\phi\in (0,2]$,  we let $\phi(E)=\phi$. In the other case, if $Z(E)=0$, we let $\phi(E)$ to be any real number in $(0,2]$ and $\phi(E[1])\equiv\phi(E)+1\ (mod \ 2\mathbb{Z})$. In the end, we will show that this is well defined, i.e. all the possible stability conditions defined in this way are equivalent to each other.
	
	For the real decomposition, one can show that for any two indecomposable objects $E',F'\in\widetilde{\mathcal{C}}$, we define $$Hom_{\mathcal{C}}(\pi(E'),\pi(F'))=\bigoplus_{k\in \mathbb{Z}} Hom_{\widetilde{\mathcal{C}}}(E', F'[2k]) .$$ This defines the same set of  homogeneous morphisms as in $\sigma$. We need to assign appropriate degree of each homogeneous morphism.

	Let $f:E'\rightarrow F'$ be a nonzero morphism between two indecomposable objects in $\widetilde{\mathcal{C}}$. In the case when $E'\in \mathcal{P}(\phi_1), F'\in\mathcal{P}(\phi_2)$, we define $q(f)=\phi_2-\phi_1$, which is a positive real number by definition. In other cases, let us take $S$ to be the set consisting of  unstable indecomposable objects in $\widetilde{\mathcal{C}}$ such that, any unstable indecomposable object in $\widetilde{\mathcal{C}}$ is isomorphic to $A'[k]$, where $A'\in S$ and $k\in\mathbb{Z}$, and for any two different objects $A', B'\in S$, we have that $A'$ is not isomorphic to $B'[k]$ for any $k\in\mathbb{Z}$. 
	
	In the case when $E'\in\mathcal{P}(\phi_1)$, and $F'\simeq A'[k]$ where $A'\in S$. Let us take $A'\xrightarrow{p_n} Q_n'$ to be the natural morphism from $A'$ to its last Harder-Narasimhan factor $Q_n'$. Then we have the following diagram $$\begin{tikzcd}
		E'\in \mathcal{P}(\phi_1)\arrow{r}{f} & F'\simeq A'[k]\arrow{r}{p_n[k]} &Q_n'[k]\in\mathcal{P}(k+\phi_2)
	\end{tikzcd}$$
	We define $$q(\pi(p_{n,i}))=\phi_2-\phi(A'),\ \ q(\pi(f))=k+\phi(A')-\phi_1,$$ where $p_{n,i}$ is a summand of the morphism of $p_n$.
	
	The case when $F'\in\mathcal{P}(\phi_2)$ and  $E'\simeq A'[k]$ where $A'\in S$ is similar. Indeed, we have the following diagram in $\widetilde{\mathcal{C}}$ $$\begin{tikzcd}
		E'\arrow{r}{f} \arrow{d}{p_n[k]}& F' \in\mathcal{P}(\phi_2)\\ Q_n'[k]\in\mathcal{P}(\phi_1+k) &
	\end{tikzcd}$$ where $p_n:A'\rightarrow Q_n'$ is the natural morphism from $A'$ to its last Harder-Narasimhan factor $Q_n'\in\mathcal{P}(\phi_1)$. We define $$q(\pi(p_{n,i}))=\phi_1-\phi(A'),\ \ q(\pi(f))=\phi_2-\phi(A')-k.$$

The last case is when $E'\simeq A'[k_1]$ and $F'\simeq B'[k_2]$, we have the following diagram $$\begin{tikzcd}
	E'\arrow{r}{f} \arrow{d}{p_n[k_1]}& F' \arrow{d}{p_m[k_2]} \\ Q_n'[k]\in\mathcal{P}(\phi_1+k_1) & Q_m'\in\mathcal{P}(\phi_2+k_2)
\end{tikzcd}$$ where $p_n:A'\rightarrow Q_n', \ p_m:B'\rightarrow Q_m'$ are the natural morphisms from $A', B'$ to their last Harder-Narasimhan factors $Q_n'\in\mathcal{P}(\phi_1), Q_m'\in\mathcal{P}(\phi_2)$ respectively. We define $$q(\pi(f))=k_2-k_1+\phi(B')-\phi(A').$$

This gives the degree function of the real decomposition on $\mathcal{C}$. Thus, we got our data $\sigma_1=(\mathcal{Q}, Z,\phi,q)$ on $\mathcal{C}$. Before proving that $\sigma$ is indeed a stability condition on $\mathcal{C}$ and our map $\mathbb{P}$ is well defined, let us state an easy consequence of the construction of degrees. 

For any connecting path $$l:E'\in\mathcal{P}(\phi_1)\dashrightarrow F'\in\mathcal{P}(\phi_2)$$ in $\widetilde{\mathcal{C}}$ (with respect to $\sigma_1$), we have $$q(\pi(l))=\phi_2-\phi_1.$$This follows directly from the construction of degree function $q$, we leave the direct check to the reader. We call such property as the invariance of semi-stable path.

There are several things to check for the construction of $\sigma_1$, we list them below. \begin{enumerate}
	\item The data $\sigma_1=(\mathcal{Q},Z,\phi, q)$ is a stability condition on $\mathcal{C}$,
	\item the charge triple $(Z,\phi, q)$ is deformation equivalent to the charge triple of $\sigma$,
	\item and the equivalent class of $\sigma_1$ is independent of the choices of $\phi(E)$ for $Z(E)=0$ and the set $S$.
\end{enumerate}

For (1), most conditions in Definition \ref{stability condition on cyclic categories} are direct consequences of the construction. We only need to check that the Harder-Narasimhan diagram is  connective and liftable. We prove the connectivity by the induction on the number of Harder-Narasimhan factors. 

Firstly, if $E'$ is semi-stable in $\widetilde{\mathcal{C}}$, the connectivity is obvious. Assume that the connectivity condition is true for any indecomposable object $E'$ with $n-1$ Harder-Narasimhan factors.  We consider an indecomposable object $E'$ with $n$ Harder-Narasimhan factors. Let the following distinguished triangle $$E'_{n-1}\xrightarrow{f_n} E'_n=E'\xrightarrow{ p_n} Q_n' \xrightarrow{} E'_{n-1}[1]$$ be the last triangle in the Harder-Narasimhan filtration of $E'$ with respect to $\sigma_1'$. We claim that every summand of $f_n$ or $p_n$ is non-trivial. The proof is essentially included in the proof of Lemma \ref{important lemma}. We prove the case of $f_n$ for readers' convenience. If there is an indecomposable summand $E_{n-1,1}'$ of $E_{n-1}'$ with $f_n|_{E_{n-1,1}'}=0$. Then we have the following commutative diagram $$\begin{tikzcd}
	E_{n-1,1}'\arrow{r}{0} \arrow{d}{i}& E_n' \arrow{r} \arrow{d}{id}& E_{n-1,1}'[1]\oplus E_n'\arrow{r} \arrow{d}& E_{n-1,1}'[1] \arrow{d}{i[1]}\\ E_{n-1}'\arrow{r}{f_n} & E_n' \arrow{r}{p_n} & Q_n'\arrow{r} &E_{n-1}'[1]
\end{tikzcd}$$ This implies that $E_{n-1,1}'[1]$ is a direct summand of $Q_n'$, but this is impossible unless $E_{n-1,1}'=0$, as every Harder-Narasimhan factor of $E_{n-1}$ had bigger phase the phase of $Q_n'$. Thus the claim is proved. The connectivity is also proved, as the Harder-Narasimhan diagram of any indecomposable summand of $E_{n-1}'$ is connective by induction, and every such  summand is connected to $E'$ and the summands of $Q_n'$.

The liftability of Harder-Narasimhan diagram follows from the connectivity and the invariance of semi-stable path. Indeed, let $l_1,l_2:A' \dashrightarrow B'$ be two connecting path in the Harder-Narasimhan diagram of $E'$. There are connecting paths $l_3:Q_1 \dashrightarrow A'$ and $l_4:B'\dashrightarrow Q_n$ by connectivity condition. By the invariance of semi-stable path, we get $$ q(l_4\circ l_1\circ l_3)=q(l_4\circ l_2\circ l_3),$$ hence $q(l_1)=q(l_2)$ and (1) is proved.

For (2), let $l_1,l_2:E\dashrightarrow F$ be two connecting paths in $\mathcal{C}$. They can be lifted to connecting paths $l_1': E'\dashrightarrow F'$ and $l_2':E'\dashrightarrow F'[2k]$ in $\widetilde{\mathcal{C}}$ where $2k=q_0(l_2)-q_0(l_1)$ ($q_0$ is the degree function in $\sigma$). From the definition of $q$, one can easily show that $q(l_2)-q(l_1)$ also equals to $2k$. Hence (2) is proved.

For (3), it follows directly from the invariance of semi-stable path. Hence $\mathbb{P}$ is well defined.
	
By unwinding the definitions, it is easy to see that $\mathbb{L}\circ \mathbb{P}=id $ and $\mathbb{P}\circ \mathbb{L}=id$.  Hence, the theorem is proved.
\end{proof}
\begin{remarks}
	(i) The equivalence relation in Definition \ref{equivalent relations} is the reason why the phases of unstable objects in Bridgeland stability conditions are not well defined in general. 
	
	(ii) Note that in this paper, we always fix the set of homogeneous morphisms in the real decomposition. Theorem \ref{Main theorem} roughly says that the deformation of stability conditions with fixed set of homogeneous morphisms corresponds to the deformation of Bridgeland stability conditions on $\widetilde{\mathcal{C}}$. We will investigate the other deformation direction (the deformation of the set of homogeneous morphisms) in the future research, we expect that it corresponds to the deformation of the category $\widetilde{C}$ itself.
	
	(iii) It could be interesting to compare this theorem with \cite[Theorem 1]{Geometricclassificationoftotalstabilityspaces}. We also have the isomorphism $$Stab^{0,e}_{cyc}(\mathcal{T}/2\mathbb{Z})\xrightarrow{\sim} Stab(\mathcal{T})/2\mathbb{Z},$$ when the orbit category $\mathcal{T}/2\mathbb{Z}$ admits a triangulated structure. See \cite{Ontriangulatedorbitcategories} for the related notion and results. 
\end{remarks}

We end this section by gently touching another aspect of the space $Stab_{cyc}(\mathcal{C})$:  chirality of objects and morphisms.

\subsection{Chirality of morphisms and objects}
There is a natural  chirality of morphisms and objects from the definition of charge triples.

\begin{definition}
	Given a real decomposition on a cyclic category $\mathcal{C}$, let $f$ be  a homogeneous morphism, we say that $f$ is left chiral if $q(f)>0$, and $f$ is right chiral if $q(f)<0$.
	
	Let $E$ be an indecomposable object in $\mathcal{C}$, we say that $E$ is left chiral if the Maslov indices of all basic loops involving $E$ are non-negative. Similarly, $E$ is right chiral if the Maslov indices of all basic loops involving $E$ are negative. In such cases, we say that $E$ has a well-defined chirality.
\end{definition}

\begin{lemma}
	Let $\mathcal{C}$ be a $k$-linear Krull-Schmidt connective cyclic category, and $q$ be a real decomposition on $\mathcal{C}$ such that $\mathcal{C}$ is  liftable with respect to $q$. If moreover, every indecomposable object has a well-defined chirality. Then all indecomposable objects has the same chirality.
\end{lemma}
\begin{proof}
	Suppose that $E$ is a left chiral indecomposable object. For any other indecomposable object $F$, there is a connecting path between $E$ and $F$. 
	
	Since any nontrivial homogeneous morphism $f:M\rightarrow N$ between two indecomposable objects can be completed to a liftable diagram  $$\cdots \rightarrow M\xrightarrow{f} N\xrightarrow{g} Cone (f)\xrightarrow{h} M[1]\rightarrow \cdots .$$ Then every summand of $g$ and $h$ are nontrivial. As the proof has already appeared twice in the proofs of Lemma \ref{important lemma} and Theorem \ref{Main theorem}, we leave it to the reader.  This non-triviality implies that the chirality of $M$ is the same as the chirality of $N$. 
	
	Hence, by chasing the connecting path, we can conclude that $F$ is also left chiral.
\end{proof}
There is a natural involution map $\tau$ on the space $\textbf{T}(\mathcal{C})$, which is defined in the following way:

$$\tau(Z,\phi,q)=(-\bar{Z},\phi', -q),$$ where $\bar{Z}$ is the conjugate of the central charge $Z$ and $\phi'(E)+\phi(E)=1(mod \ 2\mathbb{Z})$ for any indecomposable object $E$. Obviously, this involution changes the chirality of every homogeneous morphism to its opposite.

Notice that the stability condition breaks the chirality symmetry of real decompositions. In fact, for the forgetful map $$Stab(\mathcal{C})\xrightarrow{\alpha} \textbf{T}(\mathcal{C}),$$  there are charge triples $R$ with nontrivial preimage $\alpha^{-1}(R)$ but $\alpha^{-1}(\tau(R))=\emptyset$. This will be illustrated in the next section.

\section{Stability conditions on Equivariant Matrix factorizations of $A_2$ type}\label{MF of AD type}
In this section, we will present some examples of stability conditions on a cyclic category. We start with a Schur type lemma in any $k$-linear category.

\begin{lemma}
	Let $\mathcal{T}$ be any $k$-linear category, and $E,F$ are two simple objects in $\mathcal{T}$, i.e. $Hom_{\mathcal{T}}(E,E)=Hom_{\mathcal{T}}(F,F)=k$. Then the composition $$Hom_{\mathcal{T}}(E,F)\times  Hom_{\mathcal{T}}(F,E)\rightarrow Hom_{\mathcal{T}}(E,E)\xrightarrow{\sim} k$$ is trivial unless $E$ is isomorphic to $F$.
\end{lemma}
\begin{proof}
	Suppose that the composition is not trivial, i.e. there exist morphisms $a\in Hom_{\mathcal{T}}(E,F)$ and $b\in Hom_{\mathcal{T}}(F,E)$ such that their composition $b\circ a=x\cdot id _E$, where $x\in k$ is nonzero. 
	
	We claim that $a\circ b=x\cdot id_F$. Indeed, denote $a\circ b\coloneqq y\cdot id_F$. Then $$x^2\cdot id_E=(x\cdot id_E) \circ (x\cdot id_E)=b\circ a\circ b\circ a=b\circ (y\cdot id_F) \circ a=xy\cdot id_E.$$ Hence, we get $x^2=xy$, which implies $x=y$ since $x$ is nonzero. Therefore, $a$ and $x^{-1}b$ are inverse to each other. 
\end{proof}
\begin{corollary}
	Let $\mathcal{C}$ be a $k$-linear Krull-Schmidt category, and $E$ be a simple object in $\mathcal{C}$. Then either the composition $$Hom(E,E[1])\times Hom(E, E[1])\rightarrow Hom(E,E[2])\xrightarrow{\sim}k$$ is trivial or $E\xrightarrow{\sim} E[1]$. 
\end{corollary}

\begin{remark}
	As one can see from the definition of stability conditions,  the existence of non-trivial indecomposable  objects $E\simeq E[1]$ forms an obstruction of the existence of stability conditions on $\mathcal{C}$. Unfortunately, this obstruction is  very common.
	
	However, some finite  group action may resolve such obstructions. Let us begin with weighted homogeneous polynomial $w\in \mathbb{C}[[x_1,\cdots, x_N]]$ and their Abelian automorphisms (the following definitions are taken from \cite[Section 2]{FJRWtheory}).
	
\end{remark}
	
	\begin{definition}\label{quasi-homogeneous polynomial}
		A weighted homogeneous polynomial $w\in\mathbb{C}[[x_1,\cdots,x_N]]$ is a polynomial for which there exist positive rational numbers $q_1,\cdots, q_n\in\mathbb{Q}_{>0}$, such that for any $\lambda\in\mathbb{C}^*$ $$w(\lambda^{q_1}x_1,\cdots, \lambda^{q_N}x_N)=\lambda w(x_1,\cdots ,x_N).$$ We will call $q_j$ the weight of $x_j$. We define $d$ and $n_i$ to be the unique positive integers such that $(q_1,\cdots, q_N)=(n_1/d, \cdots, n_N/d)$ with $gcd(d,n_1,\cdots,n_N)=1$.
	\end{definition}

\begin{definition}
	We call $w$ nondegenerate if \begin{enumerate}
		\item the polynomial $w$ contains non monomial of the form $x_ix_j$, for $i\neq j$ and
		
		\item the hypersurface defined by $w=0$ in weighted projective space is non-singular, or equivalently, the affine hypersurface defined by $w=0$ has isolated singularity at the origin.
		
	\end{enumerate}
\end{definition}
	\begin{definition}
		There are special groups associated with the polynomial $w$. The first one $G_w$ is defined in the following way
		
		$$G_w\coloneqq\{(\alpha_1,\cdots, \alpha_N)\in (\mathbb{C}^*)^N| w(\alpha_1x_1,\cdots, \alpha_Nx_N)=w(x_1,\cdots, x_N)\}.$$
		
		There is special element $J\in G_w$ which is defined to be $$J\coloneqq (exp(2\pi i q_1), \cdots, exp(2\pi iq_N)),$$ where the $q_i$ are the weights defined in Definition \ref{quasi-homogeneous polynomial}. For any group $G$ with $\langle J\rangle \leq G\leq G_w$, we call $G$ an admissible subgroup of  $G_w$.
	\end{definition}

Recall the following famous $ADE$ examples (see e.g. \cite{KSTmatrixfactorizations}).

\begin{example}  The simple singularities can be classified in the following way.
	\begin{itemize}
		\item $A_n: w=x^{n+1}, n\geq 1;$
		\item  $D_n:  w=x^{n-1}+xy^2, n\geq 4$;
		\item $E_6: w=x^3+y^4$;
		\item $E_7: w=x^3+xy^3$;
		\item $E_8: w=x^3+y^5$.
	\end{itemize}
	
	\begin{figure}[p]\label{Figure 2}
		\centering
		\includegraphics[scale=1]{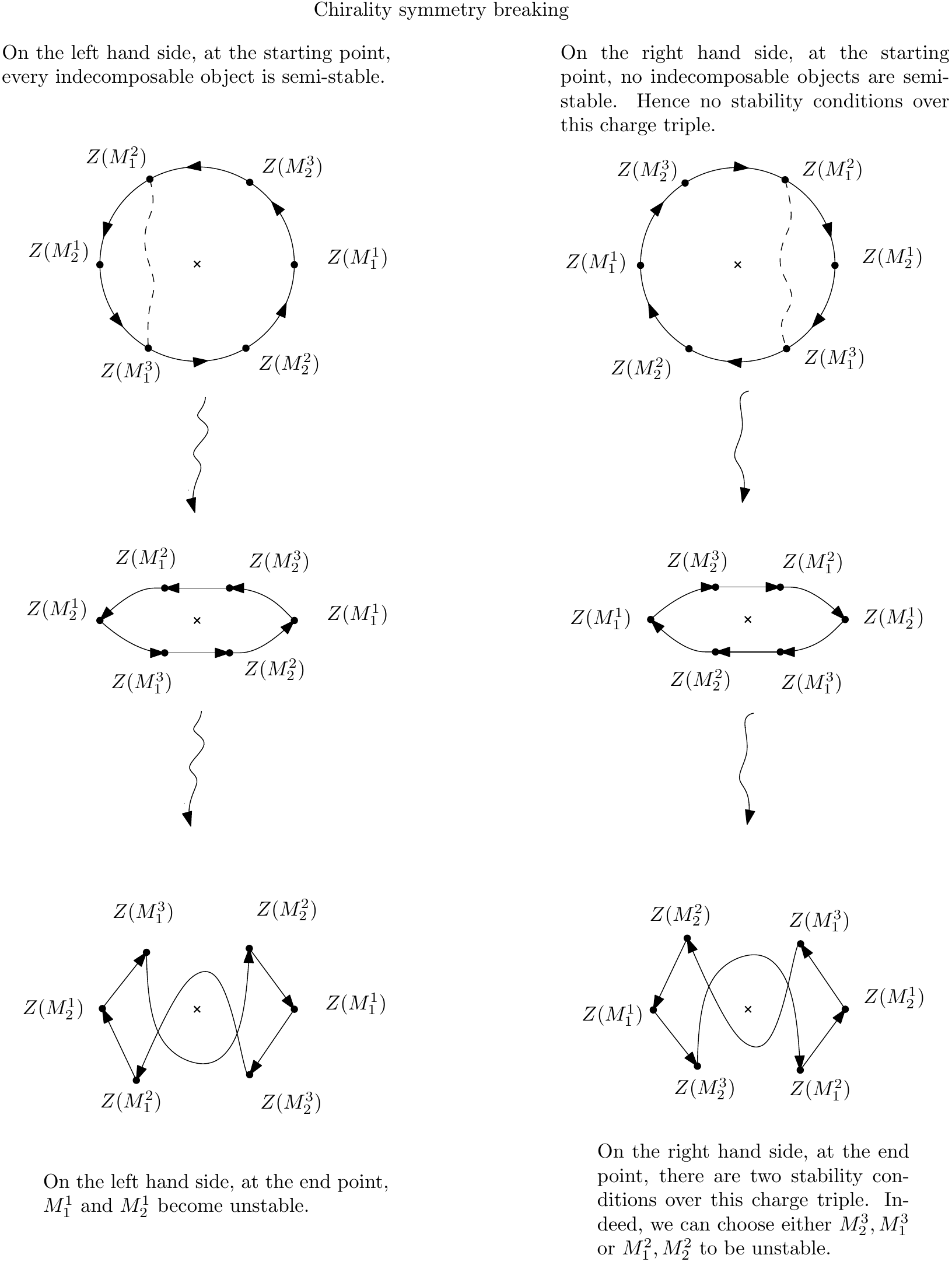}
		\caption{}
	\end{figure}
	By \cite{CMmoduleoverhypersurfaceI} and \cite{CMmodulesoverhypersurfaceII}, we know that the category $HMF(R, w)$ has only  finitely many indecomposable objects if and only if $w$ is of simple singularity.
\end{example}

 In the following, we present some examples of stability conditions on the homotopy category of $\mathbb{Z}/3\mathbb{Z}$-equivariant  matrix factorizations of $x^3$. And we use  results from the paper \cite[Section 4.6, Section 7.1]{StabilityofLGbranes}  as our starting point, readers should consult these two sections for the details.

\begin{example}
	We draw pictures of deforming stability conditions on the homotopy category of  $\mathbb{Z}/3\mathbb{Z}$-equivariant matrix factorizations of $A_2$ case. In the top left corner of Figure 2, we start with Walcher's point as described in \cite{StabilityofLGbranes}. In the deformation procedure, we keep the central charges $Z(M_1^1)$ and $Z(M_2^1)$ unchanged, move $Z(M_1^2)$ along the dotted path to $Z(M_1^3)$ (central charge of other indecomposable objects change correspondingly), and imagine that there is a pillar standing at the origin such that the homogeneous morphisms could go around the pillar but could not pass through the pillar along the deformation.
	
		\begin{figure}[ht]\label{Figure 3}
		\centering 
		\includegraphics[scale=1]{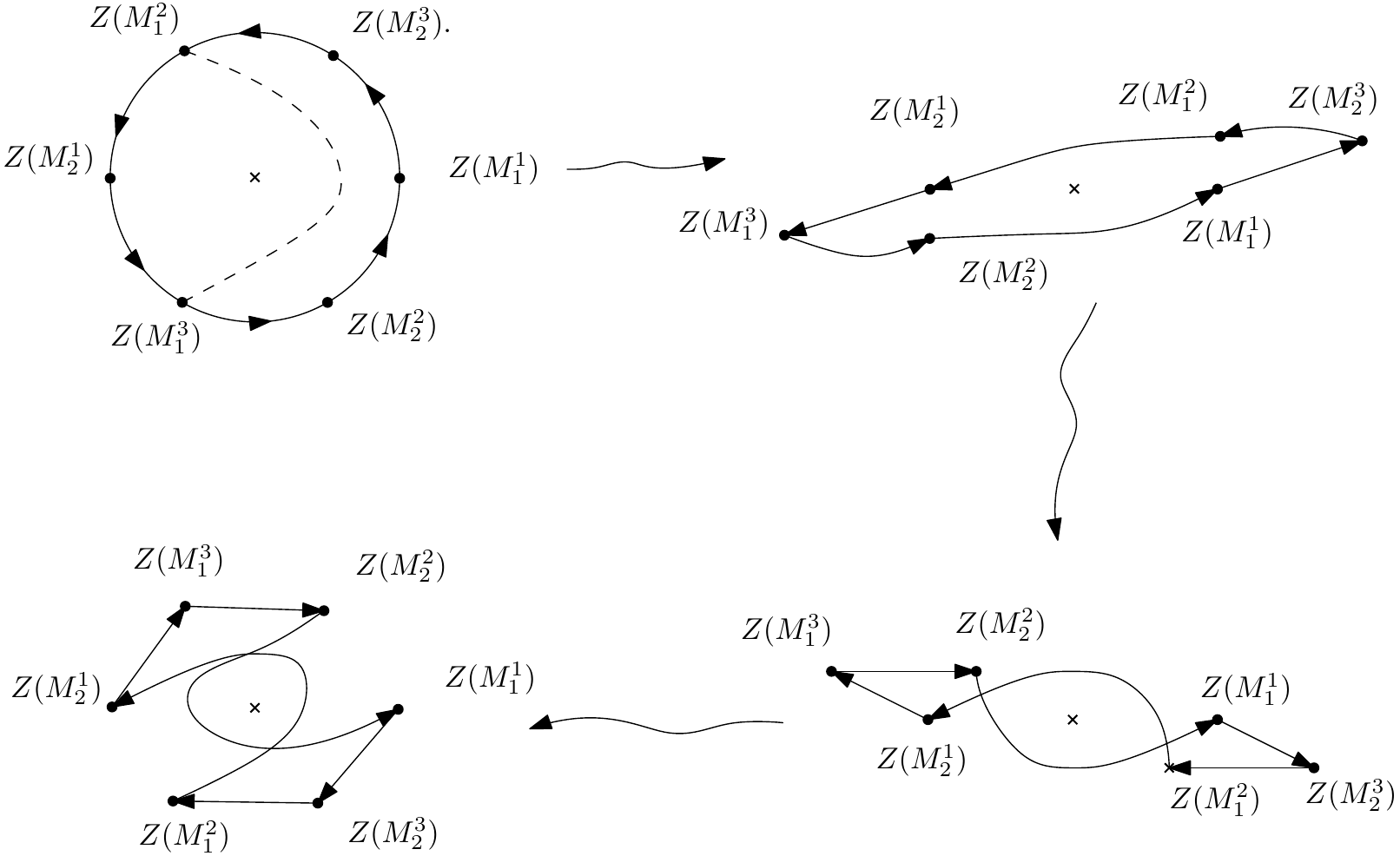}
		\caption{}
	\end{figure}
	
	In Figure 2, we observe the phenomenon of chirality symmetry breaking. In fact, the charge triples on the left hand side and right hand side are mirror symmetric to each other. But the stability conditions ar not symmetric, as we can see in Figure 1, our definition of stability conditions has an implicate orientation, which breaks the symmetry of charge triples.

	Also note that, the charge triples on the left hand side are liftable, while the charge triples on the right hand side are not.

Given a stability condition $\sigma$ on $\mathcal{C}$, if the central charge of any nontrivial indecomposable objects is nonzero, we will call $\sigma$ a strong stability condition, and use $Stab_{cyc}^s(\mathcal{C}), Hom^s(\Lambda,\mathbb{C})$ to denote the sets of strong stability condition on $\mathcal{C}$ and its associated central charges. The Figure 3 illustrates the nontrivial monodromy of the map $$Stab_{cyc}^s(\mathcal{C}) \rightarrow Hom^s(\Lambda,\mathbb{C}) \ (\subset \mathbb{C}^{rank(\Lambda)}).$$   
Indeed, if we compare Figure 3 with the left hand side of Figure 2, we see that different paths result in different stability conditions. It is easy to see that $$H_1H_2H_3=id$$ in $Stab^s_{cyc}(\mathcal{C})$, where $H_1,H_2,H_3$ are the monodromies corresponding to 3 generators of $$\pi_1((\mathbb{C}^*)^2 \backslash\{z_1+z_2=0\}).$$Moreover, if we go to the lifted space, we get $H_1H_2H_3=[2]$.

\end{example}

\bibliographystyle{alpha}
\bibliography{bibfile}

\end{document}